\documentclass[11pt]{amsart}

\usepackage[margin=1.25in]{geometry}
\usepackage[utf8]{inputenc}
\usepackage[T1]{fontenc}
\usepackage{lmodern}
\usepackage{microtype}
\usepackage{needspace}

\usepackage{amsmath,amssymb,amsthm,mathtools}
\usepackage{mathrsfs}
\usepackage{enumitem}
\usepackage[colorlinks=true,
            linkcolor=blue,
            citecolor=blue,
            urlcolor=blue]{hyperref}

\numberwithin{equation}{section}

\newtheorem{theorem}{Theorem}[section]
\newtheorem{proposition}[theorem]{Proposition}
\newtheorem{lemma}[theorem]{Lemma}
\newtheorem{corollary}[theorem]{Corollary}

\theoremstyle{definition}
\newtheorem{definition}[theorem]{Definition}

\theoremstyle{remark}
\newtheorem{remark}[theorem]{Remark}

\newcommand{\C}{\mathbb C}
\newcommand{\T}{\mathbb T}

\newcommand{\dd}{\,d}
\newcommand{\e}{\mathrm e}

\newcommand{\esssup}{\operatorname*{ess\,sup}}

\newcommand{\Lpqa}[3]{L_{#1}^{#2,#3}}
\newcommand{\Fpqa}[3]{\mathcal F_{#1}^{#2,#3}}

\newcommand{\norm}[1]{\left\lVert #1\right\rVert}
\newcommand{\abs}[1]{\left\lvert #1\right\rvert}

\title[Fock projections between mixed-norm spaces]
{Fock Projections between Gaussian Mixed-Norm Spaces:
Full-Range Boundedness and Compactness}

\author{Xiang Fang}
\address{National Yang Ming Chiao Tung University, Taiwan}
\email{xfang@nycu.edu.tw}

\author{Pham Trong Tien}
\address{VNU University of Science, Vietnam National University, Hanoi}
\email{phamtien@vnu.edu.vn}



\date{\today}

\begin{document}
%================================================

\begin{abstract}
The Fock--Bargmann projection \(P_t\) is initially an orthogonal projection on a Gaussian Hilbert space. We characterize the boundedness and compactness of
\[
P_t:L_{\alpha}^{p_1,q_1}\longrightarrow \mathcal F_{\beta}^{p_2,q_2}
\]
and of its associated maximal projection
\[
P_t^+:L_{\alpha}^{p_1,q_1}\longrightarrow L_{\beta}^{p_2,q_2}
\]
for all \(t,\alpha,\beta>0\) and \(0<p_1,q_1,p_2,q_2\le\infty\). Set
\[
\Delta(t;\alpha,\beta):=t^2-2\beta t+\alpha\beta,
\qquad
d(p,q):=\frac1q-\frac1p.
\]
For both operators, boundedness holds exactly when \(p_1,q_1\ge1\) and either
\(\Delta(t;\alpha,\beta)<0\), or
\(\Delta(t;\alpha,\beta)=0\) with
\[
q_1\le q_2,
\qquad
d(p_2,q_2)\le d(p_1,q_1).
\]
Both operators are compact exactly when \(p_1,q_1\ge1\) and either
\(\Delta(t;\alpha,\beta)<0\), or
\(\Delta(t;\alpha,\beta)=0\) with both of the above boundary inequalities strict. Thus boundedness forces the source exponents into the Banach range, whereas when \(\Delta(t;\alpha,\beta)<0\), every target exponent is admissible, including quasi-Banach ones. Annular discretization reduces the problem to drifted Gaussian sequence operators, whose comparison with diagonal sequence operators yields the endpoint criteria and quantitative tail estimates. A dilation--multiplier transfer gives corresponding results for Gaussian-twisted Fock operators, including Berezin transforms.
\end{abstract}

\maketitle
%\tableofcontents

%================================================
\section{Introduction and main results}
\label{sec:introduction}
%================================================

%------------------------------------------------
\subsection{Analytic projections and the Fock-space setting}
\label{subsec:analytic-projections}
%------------------------------------------------

A basic structural problem in analytic function-space theory is to
determine when the orthogonal projection onto a holomorphic subspace,
initially defined in a Hilbert space, extends boundedly to the
surrounding non-Hilbert scales. Such an extension is more than an
estimate for a reproducing kernel: it identifies the analytic space as
a canonically complemented subspace and provides the projection
mechanism underlying reproducing formulas, duality, and the mapping
theory of analytic operators. The Riesz--Szegő and Bergman projections
are the classical models. Their \(L^p\), weighted \(L^p\), and
geometric theories show that projection boundedness records a delicate
interaction among holomorphy, the ambient measure, and the underlying
geometry; see
\cite{Duren1970,DurenSchuster2004,HedenmalmKorenblumZhu2000} and, for
more recent developments,
\cite{CalziPeloso2023,DengHuangZhaoZheng2001,
GreenWagner2024,LanzaniStein2012}.

\smallskip
\noindent
The Fock, or Bargmann, projection is the corresponding analytic
projection in the Gaussian setting; see
\cite{Folland1989,Zhu2012}. Its geometry differs essentially from that
of the circle or the unit disc. The underlying domain is the whole
complex plane, the reproducing kernel grows exponentially, and
Gaussian weights provide the compensating confinement at infinity.
Consequently, once the kernel parameter and the source and target
Gaussian weights are allowed to vary independently, boundedness is
governed by an exact balance between exponential kernel growth and
Gaussian decay. Mixed-norm spaces, introduced by Benedek and Panzone
\cite{BenedekPanzone1961}, refine this problem further by separating
angular integrability on circles from radial aggregation across the
Gaussian measure. The resulting off-diagonal problem therefore
measures simultaneously the transfer of Gaussian decay and of two
distinct forms of integrability.

\smallskip
\noindent
This formulation is also the natural setting for compactness. An
idempotent projection onto an infinite-dimensional subspace cannot be
compact on the same ambient space, whereas a projection acting between
different Gaussian or mixed-norm scales may be compact precisely when
it produces a strict gain. Recent work has developed the mixed-norm
boundedness theory of Fock projections and obtained the complete
off-diagonal classification when all source and target exponents lie
in the Banach range; see
\cite{Liu2024,LiuShiWei2027}. That theory does not decide whether
quasi-Banach source exponents are admissible, does not cover
quasi-Banach target spaces, and does not provide the corresponding
full compactness classification. These are not merely formal
extensions: below exponent one the usual convexity arguments are no
longer available, while at exact Gaussian balance the problem is
controlled by finer angular and radial effects that are suppressed
when a strict Gaussian margin is present.

\smallskip
\noindent
The purpose of this paper is to settle the boundedness and compactness
problems over the entire exponent scale. The resulting classification
has a simple structural form. Boundedness forces both source exponents
to lie in the Banach range. Subject to this necessary condition, a
strict Gaussian margin permits every target exponent, including the
full quasi-Banach range. At exact Gaussian balance, boundedness is
governed by two sharp comparisons between the angular and radial
exponents, and compactness holds exactly when both comparisons become
strict. Thus the original seven-parameter mapping problem reduces to
an interior--boundary dichotomy together with two endpoint
inequalities. The same classification is stable under suitable
dilations and Gaussian multipliers, leading also to results for
Gaussian-twisted Fock operators, Berezin transforms, and more general
Fock-kernel operators.

%------------------------------------------------

\subsection{Mixed-norm setting and main results}
\label{subsec:mapping-problem}

Mixed norms distinguish angular integrability on each circle from the radial accumulation of those means. This distinction disappears in the classical diagonal case $p=q$, but becomes decisive for off-diagonal projection problems.
In this paper, we equip \(\T\) with normalized Lebesgue measure
\(\frac{\dd\theta}{2\pi}\) and, for a measurable function \(f\) on \(\C\)
and \(r>0\), define its \(p\)-th integral mean by
\[
M_p(f,r)
:=
\begin{cases}
\displaystyle
\left(
\frac{1}{2\pi}
\int_0^{2\pi}
\abs{f(re^{i\theta})}^p\,\dd\theta
\right)^{\frac{1}{p}},
& 0<p<\infty, \\[4mm]
\displaystyle
\esssup_{0\leq\theta<2\pi}
\abs{f(re^{i\theta})},
& p=\infty.
\end{cases}
\]
Let \(\alpha>0\) and \(0<p,q\leq\infty\). For a measurable function
\(f\) on \(\C\), set
\[
\norm{f}_{p, q, \alpha}
:=
\begin{cases}
\displaystyle
\left(
q\alpha
\int_0^\infty
M_p(f,r)^q
\e^{-\frac{q\alpha}{2}r^2}
r\,\dd r
\right)^{\frac{1}{q}},
& 0<q<\infty, \\[4mm]
\displaystyle
\esssup_{r>0}
M_p(f,r)\e^{-\frac{\alpha}{2}r^2},
& q=\infty.
\end{cases}
\]
The Gaussian mixed-norm space \(\Lpqa{\alpha}{p}{q}\) consists of all
measurable functions \(f\) on \(\C\) for which
\(
\norm{f}_{p, q, \alpha}<\infty.
\)
The corresponding mixed-norm Fock space is
\[
\Fpqa{\alpha}{p}{q}
:=
\Lpqa{\alpha}{p}{q}\cap H(\C),
\] 
where \(H(\C)\) denotes the space of entire functions on \(\C\). We equip \(\Fpqa{\alpha}{p}{q}\) with the inherited quasi-norm. When \(p=q\),
\(\mathcal F_\alpha^{p,p}\) coincides isometrically, under the present
normalization, with the classical Fock space \(\mathcal F_\alpha^p\); we write
\[
\norm{f}_{p,\alpha}
:=
\norm{f}_{p,p,\alpha}.
\] 

\medskip
\noindent
For \(t>0\), let
\[
\dd\lambda_t(w)
:=
\frac{t}{\pi}\e^{-t|w|^2}\,\dd A(w),
\]
where \(\dd A\) denotes Lebesgue area measure on \(\C\). The Hilbert
Fock space \(\mathcal F_t^2\) is a closed subspace of
\(
L_t^2:=L^2(\C,\dd\lambda_t)
\)
with reproducing kernel
\[
K_t(z,w):=\e^{tz\overline w}.
\]
The orthogonal projection \(P_t:L_t^2\to F_t^2\) is given by
\[
(P_tf)(z)
:=
\int_{\mathbb C}
K_t(z,w)f(w)\,d\lambda_t(w),
\]
and the associated maximal Fock projection is
\[
(P_t^+f)(z)
:=
\int_{\mathbb C}
\abs{K_t(z,w)}f(w)\,d\lambda_t(w).
\]

\medskip
\noindent
For the off-diagonal problem considered below, we initially interpret
these kernel formulas on bounded compactly supported measurable
functions. We determine when they define bounded or compact operators
\[
P_t,\ P_t^+:
L_\alpha^{p_1,q_1}
\longrightarrow
L_\beta^{p_2,q_2},
\]
where
\[
t,\alpha,\beta>0,
\qquad
0<p_1,q_1,p_2,q_2\leq\infty.
\]
Thus the problem is off-diagonal in the Gaussian parameters, the
angular exponents, and the outer radial exponents.

%------------------------------------------------
\medskip
\noindent
For \(t,\alpha,\beta>0\) and \(0<p,q\leq\infty\), set
\[
\Delta(t;\alpha,\beta)
:=
t^2-2\beta t+\alpha\beta,
\qquad
d(p,q)
:=
\frac1q-\frac1p,
\]
with the convention \(\frac{1}{\infty}=0\). 
We refer to
\(
\Delta(t;\alpha,\beta)<0
\)
as the \emph{strict Gaussian region}, and to
\(
\Delta(t;\alpha,\beta)=0
\)
as the \emph{Gaussian boundary}.
\begin{theorem}[Boundedness classification]
\label{thm:main-boundedness}
Let \(t,\alpha,\beta>0\) and
\(
0<p_1,q_1,p_2,q_2\leq\infty.
\)
Then the following assertions are equivalent:
\begin{enumerate}[label=\textup{(\roman*)}]
\item
The maximal Fock projection \(
P_t^+:
L_\alpha^{p_1,q_1}
\longrightarrow
L_\beta^{p_2,q_2}
\) is bounded.

\item
The Fock projection \(
P_t:
L_\alpha^{p_1,q_1}
\longrightarrow
L_\beta^{p_2,q_2}\) is bounded.

\item
One has \(p_1,q_1\geq1\), and either
\[
\Delta(t;\alpha,\beta)<0,
\]
or
\[
\Delta(t;\alpha,\beta)=0,
\qquad
q_1\leq q_2,
\qquad
d(p_2,q_2)\leq d(p_1,q_1).
\]
\end{enumerate}
\end{theorem}

\begin{theorem}[Compactness classification]
\label{thm:main-compactness}
Let \(t,\alpha,\beta>0\) and
\(
0<p_1,q_1,p_2,q_2\leq\infty.
\)
Then the following assertions are equivalent:
\begin{enumerate}[label=\textup{(\roman*)}]
\item
The maximal Fock projection \(
P_t^+:
L_\alpha^{p_1,q_1}
\longrightarrow
L_\beta^{p_2,q_2}
\) is compact.

\item
The Fock projection \(
P_t:
L_\alpha^{p_1,q_1}
\longrightarrow
L_\beta^{p_2,q_2}\) is compact.

\item
One has \(p_1,q_1\geq1\), and either
\[
\Delta(t;\alpha,\beta)<0,
\]
or
\[
\Delta(t;\alpha,\beta)=0,
\qquad
q_1 <  q_2,
\qquad
d(p_2,q_2) < d(p_1,q_1).
\]
\end{enumerate}
\end{theorem}

%------------------------------------------------

\subsection{Relation with previous work and contributions}

The classical $L^p$ theory of the Fock and maximal Fock projections is
presented in \cite[Subsection~2.2]{Zhu2012}. Hou and Liu subsequently
studied the $L^p$-to-$L^q$ boundedness and compactness of both operators
for $1\leq p,q<\infty$; see \cite{HouLiu2023}. Boundedness on
mixed-norm spaces was considered by Liu \cite{Liu2024}. More recently,
Liu, Shi, and Wei obtained the complete off-diagonal boundedness
classification in the Banach mixed-norm range
\[
1\leq p_1,q_1,p_2,q_2\leq\infty;
\]
see \cite[Theorem~1.1]{LiuShiWei2027}. In this range,
Theorem~\ref{thm:main-boundedness} agrees with their result.

\smallskip
\noindent
Theorem~\ref{thm:main-boundedness} completes the boundedness
classification over the full exponent scale. It reveals a sharp
source--target asymmetry: boundedness itself forces
$p_1,q_1\geq1$, whereas in the strict Gaussian region every target
exponent is admissible, including quasi-Banach ones. At the Gaussian
boundary, the remaining conditions are exactly
\[
q_1\leq q_2 
\qquad \text{and} \qquad
d(p_2,q_2)\leq d(p_1,q_1).
\]
Thus the seven-parameter mapping problem is organized by the sign of
$\Delta(t;\alpha,\beta)$ and, only at exact Gaussian balance, by two
explicit comparisons of mixed-norm integrability.

\smallskip
\noindent
Theorem~\ref{thm:main-compactness} gives the corresponding complete
compactness classification for both $P_t$ and $P_t^+$. Under the
necessary condition $p_1,q_1\geq1$, compactness holds throughout the
strict Gaussian region, while at the Gaussian boundary both mixed-norm
inequalities governing boundedness must become strict:
\[
q_1<q_2
\qquad \text{and} \qquad
d(p_2,q_2)<d(p_1,q_1).
\]
In particular, when specialized to the classical Gaussian
$L^p$-to-$L^q$ scale, the strict defect condition rules out
compactness at the Gaussian boundary. This differs from the boundary
alternative stated in \cite[Theorem~3.4]{HouLiu2023}; see
Remark~\ref{rem:hou-liu-boundary} below for a detailed comparison.

\smallskip
\noindent
The proof also yields quantitative source-tail estimates, with Gaussian
decay in the strict Gaussian region and polynomial decay at the
Gaussian boundary. Finally, a dilation--multiplier reduction transfers
the boundedness and compactness classifications to Gaussian-twisted
Fock operators, Berezin transforms, and more general Fock-kernel
operators.

\subsection{Main difficulties and proof strategy}

The full-range problem involves two difficulties that are absent from
the Banach mixed-norm theory. First, source exponents below one must be
excluded rather than treated by convexity arguments. Second,
quasi-Banach target exponents are genuinely admissible, so the
sufficiency proof must avoid direct use of Minkowski's inequality in
those exponents. At the Gaussian boundary there is an additional
endpoint difficulty: the Gaussian decay is exactly balanced, and the
remaining angular and radial powers become decisive.

\smallskip
\noindent
For necessity, we use four test families adapted to the different
obstructions. Thin radial shells and angular sectors force
$q_1\geq1$ and $p_1\geq1$, respectively. Gaussian-weighted monomials
detect the condition $\Delta(t;\alpha,\beta)\leq0$ and, at the
Gaussian boundary, the ordering $q_1\leq q_2$. Gaussian-weighted
exponential packets yield the remaining defect inequality
\[
d(p_2,q_2)\leq d(p_1,q_1).
\]
After normalization, the monomial and exponential families also form
escaping bounded families. Their images converge locally uniformly to
zero, while the target quasi-norms fail to vanish at the corresponding
endpoint equalities; this gives the strict conditions required for
compactness.

\smallskip
\noindent
For sufficiency, we first perform the angular convolution estimate at
a Banach exponent and then pass to the desired target angular exponent.
Annular discretization reduces the radial estimate to a drifted Gaussian
sequence operator of the form
\[
(\mathcal Tx)_\ell
=
\sum_{k=0}^\infty
\e^{-\gamma\abs{\ell-\tau k}^2}
\e^{-\eta k^2}
(1+k)^\sigma x_k,
\qquad
\tau=\frac{t}{\beta},
\qquad
\eta=-\frac{\Delta(t;\alpha,\beta)}{2\beta},
\]
with $\gamma>0$. A comparison with a diagonal sequence operator gives
the sharp boundedness conditions. When $\Delta(t;\alpha,\beta)<0$, the
residual Gaussian factor $\e^{-\eta k^2}$ absorbs every polynomial
loss; at the Gaussian boundary it disappears, leaving precisely the
two mixed-norm endpoint inequalities. The corresponding tail estimates,
combined with compact source truncations, yield compactness of $P_t^+$,
while a compactness-transfer argument gives the result for $P_t$.

\smallskip
\noindent
Section~\ref{sec:preliminaries} develops the mixed-norm preliminaries,
annular discretization, and the boundedness, compactness, and tail
estimates for drifted Gaussian sequence operators.
Section~\ref{sec:boundedness} combines the four testing arguments with
the unified shell estimate to prove
Theorem~\ref{thm:main-boundedness}. Section~\ref{sec:compactness}
uses normalized escaping families and source truncation to prove
Theorem~\ref{thm:main-compactness}. Section~\ref{sec:gaussian-twisted}
transfers the boundedness and compactness classifications to
Gaussian-twisted Fock operators, Berezin transforms, and Fock-kernel
operators.

\medskip
\noindent
\textbf{Notation and conventions.}
We write
\[
\mathbb N=\{1,2,\ldots\},
\qquad
\mathbb N_0=\{0,1,2,\ldots\},
\]
and use the convention \(\frac{1}{\infty}=0\). For \(1\leq r\leq\infty\),
the conjugate exponent \(r'\) is defined by
\[
\frac1r+\frac1{r'}=1.
\]
The characteristic function of a set \(E\) is denoted by \(\chi_E\).

\smallskip
\noindent
Throughout the paper, \(C>0\) denotes a constant that may change from
one occurrence to the next. Unless otherwise indicated, implicit
constants may depend only on the fixed structural parameters under
consideration, such as the Gaussian parameters and mixed-norm
exponents, but are independent of the functions and free variables
appearing in the estimates. We write \(A\lesssim B\) if \(A\leq CB\),
and \(A\asymp B\) if \(A\lesssim B\) and \(B\lesssim A\).

\smallskip
\noindent
Dependence on an auxiliary parameter is displayed explicitly, for
example by \(C_R\), \(C_N\), or \(\lesssim_R\). In particular,
constants in estimates involving shell or truncation indices are
independent of those indices unless otherwise stated.

%================================================

\section{Preliminaries and auxiliary estimates}
\label{sec:preliminaries}

This section develops the analytic and discrete tools used in the
boundedness and compactness arguments. The main device is an annular
discretization of the radial mixed quasi-norm, which converts estimates
for the Fock projections into weighted sequence estimates over unit
shells. The resulting discrete operators are drifted Gaussian matrices;
their boundedness, compactness, and tail estimates are studied in
Subsection~\ref{subsec:drifted-gaussian-operators}.

%------------------------------------------------

\subsection{Mixed-norm spaces and Fock projection formulas}
\label{subsec:mixed-norm-tools}
%------------------------------------------------
We first collect the basic mixed-norm facts, compactness criteria, and
Fock projection formulas needed later. We then establish an annular
estimate that controls the radial mixed quasi-norm by a weighted
sequence of unit-shell suprema. This estimate provides the passage
from the continuous projection problem to the sequence operators
studied in the next subsection.

\medskip
\noindent
For \(0<p,q\leq\infty\), put
\[
s_{p,q}
:= \min\{1,p,q\}.
\]
The mixed quasi-norms satisfy the \(s_{p,q}\)-triangle inequality
\[
\norm{f+g}_{p,q, \alpha}^{s_{p,q}}
\leq
\norm{f}_{p,q, \alpha}^{s_{p,q}}
+
\norm{g}_{p,q, \alpha}^{s_{p,q}}.
\]
Moreover, \(\Fpqa{\alpha}{p}{q}\) is an
\(s_{p,q}\)-Banach space; see, for instance,
\cite[Lemma~2.1]{FangTien2023}.

\medskip
\noindent
The following elementary embeddings follow from the monotonicity of
the angular \(L^p\)-means and the argument in
\cite[Proposition~3.1]{BlascoGalbis2024}.

\begin{lemma}
\label{lem:mixed-norm-embeddings}
Let \(\alpha>0\).

\begin{enumerate}[label=\textup{(\alph*)}]
\item
If \(0<p_1\leq p_2\leq\infty\) and \(0<q\leq\infty\), then
\[
\Lpqa{\alpha}{p_2}{q}
\hookrightarrow
\Lpqa{\alpha}{p_1}{q}
\qquad \text{and} \qquad
\norm{f}_{p_1, q, \alpha}
\leq
\norm{f}_{p_2, q, \alpha}.
\]

\item
If \(0<p\leq\infty\) and \(0<q_1\leq q_2\leq \infty\), then
\[
\Fpqa{\alpha}{p}{q_1}
\hookrightarrow
\Fpqa{\alpha}{p}{q_2}
\qquad \text{and} \qquad
\norm{f}_{p, q_2, \alpha}
\leq C_{q_1, q_2}
\norm{f}_{p, q_1, \alpha},
\]
where 
\[
C_{q_1,q_2}
:=
\begin{cases}
\displaystyle
\left(\frac{q_2}{q_1}\right)^{\frac{1}{q_2}},
& q_2<\infty,\\[3mm]
1,
& q_2=\infty.
\end{cases}
\]
\end{enumerate}
\end{lemma}

The following standard local estimate will be used repeatedly.

\begin{lemma}
\label{lem:point-evaluation}
Let \(\alpha>0\) and \(0<p,q\leq\infty\). For every \(R>0\), there
exists a constant
\(
C=C(\alpha,p,q,R)>0
\)
such that
\[
\sup_{\abs{z}\leq R}\abs{f(z)}
\leq
C\norm{f}_{p,q,\alpha}
\]
for every \(f\in\Fpqa{\alpha}{p}{q}\). Consequently,
convergence in \(\Fpqa{\alpha}{p}{q}\) implies locally uniform
convergence on \(\C\).
\end{lemma}

\begin{proof}
By subharmonicity of \(\abs{f}^p\) when \(p<\infty\), and by the
maximum principle when \(p=\infty\),
\[
\sup_{\abs{z}\leq R}\abs{f(z)}
\lesssim
M_p(f,R+1)
\leq
\e^{\frac{\alpha}{2}(R+1)^2}
\norm{f}_{p,q,\alpha}.
\]
Here the second inequality follows from the monotonicity of
\(M_p(f,\cdot)\) and the definition of the mixed quasi-norm, with the
usual supremum interpretation when \(q=\infty\).
\end{proof}

\medskip
\noindent
For \(f\in L_\alpha^{p,q}\) and \(R>0\), define the local and tail quasi-norms, respectively, by
\[
\norm{f}_{p,q,\alpha;[0, R]}
:=
\begin{cases}
\displaystyle
\left(
q\alpha
\int_0^R
M_p(f,r)^q
\e^{-\frac{q\alpha}{2}r^2}
r\,\dd r
\right)^{\frac{1}{q}},
& 0<q<\infty,\\[4mm]
\displaystyle
\esssup_{0< r\leq R}
M_p(f,r)\e^{-\frac{\alpha}{2}r^2},
& q=\infty,
\end{cases}
\]
and
\[
\norm{f}_{p,q,\alpha;[R, \infty)}
:=
\begin{cases}
\displaystyle
\left(
q\alpha
\int_R^\infty
M_p(f,r)^q
\e^{-\frac{q\alpha}{2}r^2}
r\,\dd r
\right)^{\frac{1}{q}},
& 0<q<\infty,\\[4mm]
\displaystyle
\esssup_{r\geq R}
M_p(f,r)\e^{-\frac{\alpha}{2}r^2},
& q=\infty.
\end{cases}
\]

\begin{proposition}
\label{prop:ambient-compactness}
Let $\alpha>0$, $0<p,q\le\infty$, and let
\(\mathcal A\subset C(\mathbb C)\cap L^{p,q}_\alpha\).
Suppose that $\mathcal A$ is relatively compact in $C(\mathbb C)$
with respect to locally uniform convergence and satisfies
\[
\lim_{R\to\infty}
\sup_{f\in\mathcal A}
\|f\|_{p,q,\alpha;[R,\infty)}=0.
\]
Then $\mathcal A$ is relatively compact in $L^{p,q}_\alpha$.
Moreover, when $q<\infty$, the uniform tail condition is necessary for
relative compactness.

\smallskip
\noindent
In particular, if $\mathcal A$ is a bounded subset of $\mathcal F^{p,q}_\alpha$,
then the uniform
tail condition implies that $\mathcal A$ is relatively compact in
$\mathcal F^{p,q}_\alpha$, and the converse holds when $q<\infty$.
\end{proposition}

\begin{proof}
Let $(f_n)_n$ be a sequence in $\mathcal A$. By relative
compactness with respect to locally uniform convergence, after passing to a subsequence there
exists $f\in C(\mathbb C)$ such that
\[
f_n\longrightarrow f
\qquad\text{locally uniformly on }\mathbb C.
\]
Hence, for every $R>0$,
\[
\|f_n-f\|_{p,q,\alpha;[0,R]}\longrightarrow0.
\]
Put $s:=s_{p,q}=\min\{1,p,q\}$. Given $\varepsilon>0$, choose $R>0$
such that
\[
\sup_{g\in\mathcal A}
\|g\|_{p,q,\alpha;[R,\infty)}<\varepsilon.
\]
Fatou's lemma when \(q<\infty\), and the
definition when \(q=\infty\), show that
\[
\|f\|_{p,q,\alpha;[R,\infty)}
\le
\liminf_{n\to\infty}
\|f_n\|_{p,q,\alpha;[R,\infty)}
\le\varepsilon.
\]
Thus $f\in L^{p,q}_\alpha$, and
\[
\|f_n-f\|_{p,q,\alpha;[R,\infty)}^s
\le
\|f_n\|_{p,q,\alpha;[R,\infty)}^s
+
\|f\|_{p,q,\alpha;[R,\infty)}^s
\le 2\varepsilon^s.
\]
Combining the local and tail estimates yields
\[
\limsup_{n\to\infty}
\|f_n-f\|_{p,q,\alpha}^s
\le 2\varepsilon^s.
\]
Since $\varepsilon>0$ is arbitrary,
\[
\|f_n-f\|_{p,q,\alpha}\longrightarrow0.
\]
Hence $\mathcal A$ is relatively compact in $L^{p,q}_\alpha$.

\medskip
\noindent
Suppose now that $q<\infty$ and that $\mathcal A$ is relatively
compact. Given $\varepsilon>0$, choose $f_1,\ldots,f_N\in L^{p,q}_\alpha$
such that every $f\in\mathcal A$ satisfies
\[
\|f-f_j\|_{p,q,\alpha}<\varepsilon
\]
for some $j$. Since
\(f_j \in L^{p,q}_{\alpha} \) for \(1 \leq j \leq N\),
we may choose $R>0$ such that
\[
\max_{1\le j\le N}
\|f_j\|_{p,q,\alpha;[R,\infty)}<\varepsilon.
\]
It follows that
\[
\|f\|_{p,q,\alpha;[R,\infty)}^s
\le
\|f-f_j\|_{p,q,\alpha}^s
+
\|f_j\|_{p,q,\alpha;[R,\infty)}^s
<2\varepsilon^s.
\]
Taking the supremum over $f\in\mathcal A$ proves the necessity.

\medskip
\noindent
Finally, if $\mathcal A$ is bounded in $\mathcal F^{p,q}_\alpha$, then
Lemma~\ref{lem:point-evaluation} gives local uniform boundedness.
Montel's theorem therefore shows that $\mathcal A$ is relatively
compact  with
respect to locally uniform convergence. The analytic assertion now
follows from the ambient one.
\end{proof}

\medskip
\noindent
The preceding proposition yields the following operator-theoretic
consequence, which will be used in the compactness argument for the
Fock projections.

\begin{corollary}
\label{cor:ambient-operator-compactness}
Let $X$ be a quasi-normed linear space and let
\(
T:X\longrightarrow C(\mathbb C)\cap L^{p,q}_\alpha
\)
be linear and bounded with respect to the $L^{p,q}_\alpha$ quasi-norm.
\begin{enumerate}[label=\textup{(\alph*)}]
\item Suppose that the image under $T$ of the unit ball of $X$ is
relatively compact with respect to locally uniform convergence. If
\[
\lim_{R\to\infty}
\sup_{\norm{f}_X\leq1}
\norm{Tf}_{p,q,\alpha;[R,\infty)}=0,
\]
then $T$ is compact. When $q<\infty$, the converse also holds.

\item If $T$ is compact, $(f_n)_n$ is bounded in $X$, and
\(
Tf_n\longrightarrow0\) locally uniformly on \(\mathbb C\),
then
\[
\norm{Tf_n}_{p,q,\alpha}\longrightarrow0.
\]
\end{enumerate}

\smallskip
\noindent
In particular, if
\(
T:X\longrightarrow\mathcal F^{p,q}_\alpha
\)
is bounded, then the local relative-compactness hypothesis in
\textup{(a)} is automatic.
\end{corollary}

\begin{proof}
Part~\textup{(a)} follows by applying
Proposition~\ref{prop:ambient-compactness} to the image under $T$ of
the unit ball of $X$. When $q<\infty$, the converse follows from the
necessity part of that proposition.

\medskip
\noindent
For part~\textup{(b)}, suppose otherwise. After passing to a
subsequence, we may assume that
\[
\norm{Tf_n}_{p,q,\alpha}\geq\varepsilon
\]
for some $\varepsilon>0$ and every $n$. By compactness, after passing
to a further subsequence, we may assume that
\[
Tf_n\longrightarrow g
\qquad\text{in }L^{p,q}_\alpha
\]
for some $g\in L^{p,q}_\alpha$.
Put $s:=s_{p,q}=\min\{1,p,q\}$. For every $R>0$, the $s$-triangle
inequality gives
\[
\norm{g}_{p,q,\alpha;[0,R]}^s
\leq
\norm{g-Tf_n}_{p,q,\alpha}^s
+
\norm{Tf_n}_{p,q,\alpha;[0,R]}^s.
\]
The first term tends to zero by convergence in $L^{p,q}_\alpha$, while
the second tends to zero by locally uniform convergence. Hence
\[
\norm{g}_{p,q,\alpha;[0,R]}=0
\quad \text{for every} \quad R>0,
\]
so $g=0$ in $L^{p,q}_\alpha$. Consequently,
\[
\norm{Tf_n}_{p,q,\alpha}\longrightarrow0
\]
along the chosen subsequence, contradicting the lower bound.

\medskip
\noindent
Finally, if $T$ takes values in $\mathcal F^{p,q}_\alpha$, boundedness
of $T$ and Lemma~\ref{lem:point-evaluation} give local uniform
boundedness of the image of the unit ball. Montel's theorem then shows
that this image is relatively compact with respect to locally uniform
convergence.
\end{proof}

\medskip
\noindent
We shall also use the following standard estimates. Since
\(
M_p(z^n,r)=r^n,
\)
we have
\[
\norm{z^n}_{p,q,\alpha}
=
\norm{z^n}_{q,\alpha}
.
\]
Thus, by \cite[p.~40]{Zhu2012} and Stirling's formula,
\begin{equation}
\label{eq:monomial-norm}
\norm{z^n}_{p,q,\alpha}
\asymp
\left(\frac{n!}{\alpha^n}\right)^{\frac12}
(1+n)^{\frac{1}{2q}-\frac14},
\qquad n\in\mathbb N_0,
\end{equation}
with the convention \(\frac{1}{\infty}=0\). Moreover, by
\cite[Proposition~2.3]{LiuShiWei2027}, for \(t>0\) and \(w\in\C\),
\begin{equation}
\label{eq:exponential-kernel-norm}
\norm{\e^{t z\overline w}}_{p,q,\alpha}
\asymp
\exp\left(
\frac{t^2}{2\alpha}\abs{w}^2
\right)
(1+\abs{w})^{d(p,q)}.
\end{equation}
The implicit constants in \eqref{eq:monomial-norm} and
\eqref{eq:exponential-kernel-norm} depend only on
\(\alpha, q\) and on \(\alpha,t, p,q\), respectively.

\medskip
\noindent
We next record a few elementary formulas for the Fock projections.
Whenever the defining integral converges absolutely,
\begin{equation}
\label{eq:pointwise-domination}
\abs{P_tf(z)}
\leq
P_t^+\abs{f}(z).
\end{equation}

\smallskip
\noindent
The following identities describe the action of \(P_t\) on the test
functions used in the necessity arguments.

\begin{lemma}
\label{lem:fock-test-formulas}
Let \(t>0\) and \(\gamma<t\).

\begin{enumerate}[label=\textup{(\alph*)}]
\item
For every \(n\in\mathbb N_0\),
\[
P_t\left(
w^n\e^{\gamma\abs{w}^2}
\right)(z)
=
\left(
\frac{t}{t-\gamma}
\right)^{n+1}
z^n.
\]

\item
For every \(\lambda>0\) and \(\zeta\in\C\),
\[
P_t\left(
\e^{\lambda w\overline \zeta}
\e^{\gamma\abs{w}^2}
\right)(z)
=
\frac{t}{t-\gamma}
\exp\left(
\frac{t\lambda}{t-\gamma}z\overline \zeta
\right).
\]
\end{enumerate}
\end{lemma}

\begin{proof}
For \(h\in\mathcal F_{t-\gamma}^2\), the reproducing property gives
\[
P_t\left(
h(w)\e^{\gamma\abs{w}^2}
\right)(z)
=
\frac{t}{\pi}
\int_{\C}
h(w)\e^{tz\overline w}
\e^{- (t - \gamma)\abs{w}^2}\,\dd A(w)\\
=
\frac{t}{t - \gamma}
h\left(\frac{t}{t - \gamma}z\right).
\]
Taking \(h(w)=w^n\) and
\(h(w)=\e^{\lambda w\overline \zeta}\), respectively,  proves
\textup{(a)} and \textup{(b)}.
\end{proof}

\medskip
\noindent
We finish this subsection with an annular estimate used in the shell
arguments for \(P_t\) and \(P_t^+\). For a measurable
function \(f\) on \(\C\), set
\[
B_{p,\alpha}(f;k)
:=
\operatorname*{ess\,sup}_{k\leq r<k+1}
M_p(f,r)\e^{-\frac{\alpha}{2}r^2},
\qquad k\in\mathbb N_0.
\]
Then
\[
\norm{f}_{p,\infty,\alpha}
=
\sup_{k\geq0}B_{p,\alpha}(f;k).
\]

\begin{proposition}[Annular discretization]
\label{prop:annular-discretization}
Let \(\alpha>0\), \(0<p\leq\infty\), and \(0<q<\infty\).
Then every measurable function \(f\) on \(\C\) satisfies
\[
\norm{f}_{p,q,\alpha}^q
\lesssim
\sum_{k=0}^\infty
(1+k)B_{p,\alpha}(f;k)^q.
\]
The implicit constant depends only on \(\alpha\) and \(q\).
\end{proposition}

\begin{proof}
This follows by decomposing the radial integral over the intervals
\([k,k+1)\), \(k\in\mathbb N_0\).
\end{proof}

%------------------------------------------------
\subsection{Drifted Gaussian sequence operators}
\label{subsec:drifted-gaussian-operators}

The annular estimates for the Fock projections lead to discrete
operators whose kernels are localized near a drifted diagonal
\(\ell\approx\tau k\). We study these operators independently. The
main point is that, after transferring the target shell weight through
the Gaussian kernel, their operator quasi-norms are equivalent to those of
suitable diagonal sequence operators. The corresponding comparison
for the tail operators yields both compactness and quantitative tail
estimates.

\medskip
\noindent
For \(0<q\leq\infty\), let \(\ell^q\) denote the usual sequence
space on \(\mathbb N_0\). For a sequence
\(y=(y_\ell)_{\ell\geq0}\), define
\[
\norm{y}_{\ell_{\mathrm{sh}}^q}
:=
\begin{cases}
\displaystyle
\left(
\sum_{\ell=0}^\infty
(1+\ell)\abs{y_\ell}^q
\right)^{\frac1q},
&0<q<\infty,\\[4mm]
\displaystyle
\sup_{\ell\geq0}\abs{y_\ell},
&q=\infty.
\end{cases}
\]
For \(0<q_2<q_1\leq\infty\), define
\[
q_*
:=
\frac{1}{d(q_1,q_2)}
=
\left(
\frac1{q_2}-\frac1{q_1}
\right)^{-1}.
\]

\medskip
\noindent
We first record two elementary facts.

\begin{lemma}
\label{lem:drifted-gaussian-matrix}
Let \(\gamma,\tau>0\). The operator
\[
(Gx)_\ell
:=
\sum_{k=0}^\infty
\e^{-\gamma\abs{\ell-\tau k}^2}x_k,
\qquad \ell\in\mathbb N_0,
\]
is bounded on \(\ell^q\) for every \(0<q\leq\infty\).
\end{lemma}

\begin{proof}
The uniform summability of translated Gaussian sequences gives
\[
\sup_{\ell\geq0}
\sum_{k=0}^\infty
\e^{-\gamma\abs{\ell-\tau k}^2}
+
\sup_{k\geq0}
\sum_{\ell=0}^\infty
\e^{-\gamma\abs{\ell-\tau k}^2}
<\infty.
\]
Hence the assertion follows from Schur's test when
\(1\leq q\leq\infty\), and from the \(q\)-triangle inequality,
using the same summability with \(\gamma\) replaced by \(q\gamma\),
when \(0<q<1\).
\end{proof}

\medskip
\noindent
For a nonnegative sequence \(v=(v_k)_{k\geq0}\), consider the
diagonal operator
\[
(D_vx)_k:=v_kx_k,
\qquad k\geq0.
\]
The following elementary characterization of diagonal sequence
operators will be used repeatedly.

\begin{lemma}
\label{lem:diagonal-sequence-operator}
Let \(0<q_1,q_2\leq\infty\).

\begin{enumerate}[label=\textup{(\alph*)}]
\item
If \(q_1\leq q_2\), then
\(
D_v:\ell^{q_1}\longrightarrow\ell^{q_2}
\)
is bounded if and only if \(v\in\ell^\infty\), and
\[
\norm{D_v}_{\ell^{q_1}\to\ell^{q_2}}
=
\norm{v}_{\ell^\infty}.
\]
Moreover, \(D_v\) is compact if and only if
\(
v_k\longrightarrow0.
\)

\item
Suppose that \(q_2<q_1\).
Then
\(
D_v:\ell^{q_1}\longrightarrow\ell^{q_2}
\)
is bounded if and only if \(v\in\ell^{q_*}\), and
\[
\norm{D_v}_{\ell^{q_1}\to\ell^{q_2}}
=
\norm{v}_{\ell^{q_*}}.
\]
In this case, boundedness and compactness are equivalent.
\end{enumerate}
\end{lemma}

\begin{proof}
Suppose first that \(q_1\leq q_2\). Since
\(\ell^{q_1}\hookrightarrow\ell^{q_2}\),
\[
\norm{D_vx}_{\ell^{q_2}}
\leq
\norm{v}_{\ell^\infty}\norm{x}_{\ell^{q_2}}
\leq
\norm{v}_{\ell^\infty}\norm{x}_{\ell^{q_1}}.
\]
Testing on the coordinate vectors \(e_k\) gives the reverse
inequality and hence
\[
\norm{D_v}_{\ell^{q_1}\to\ell^{q_2}}
=
\norm{v}_{\ell^\infty}.
\]
For
\[
v_k^{[N]}:=v_k\chi_{\{k\leq N\}},
\]
the operator \(D_{v^{[N]}}\) has finite rank and
\[
\norm{D_v-D_{v^{[N]}}}_{\ell^{q_1}\to\ell^{q_2}}
=
\sup_{k>N}v_k.
\]
Thus \(v_k\to0\) implies compactness. Conversely, if
\(v_k\not\to0\), there exist \(\varepsilon>0\) and distinct indices
\(k_j\) such that \(v_{k_j}\geq\varepsilon\). The sequence
\((D_ve_{k_j})_j\) is then uniformly separated in \(\ell^{q_2}\), so
\(D_v\) is not compact.

\smallskip
\noindent
Suppose now that \(q_2<q_1\).
H\"older's inequality gives
\[
\norm{D_vx}_{\ell^{q_2}}
\leq
\norm{v}_{\ell^{q_*}}
\norm{x}_{\ell^{q_1}}.
\]
Conversely, if \(q_1<\infty\), then for every finite
\(E\subset\mathbb N_0\) such that
\[
\sum_{j\in E}v_j^{q_*}>0,
\]
the sequence
\[
x_k
:=
\frac{
v_k^{\frac{q_*}{q_1}}\chi_E(k)
}{
\left(\sum_{j\in E}v_j^{q_*}\right)^{\frac{1}{q_1}}
}
\]
satisfies \(\norm{x}_{\ell^{q_1}}=1\) and
\[
\norm{D_vx}_{\ell^{q_2}}
=
\left(
\sum_{k\in E}v_k^{q_*}
\right)^{\frac{1}{q_*}}.
\]
When \(q_1=\infty\), the same conclusion follows by testing on
\(x=\chi_E\). Taking the supremum over finite sets \(E\) proves
\[
\norm{D_v}_{\ell^{q_1}\to\ell^{q_2}}
=
\norm{v}_{\ell^{q_*}}.
\]
Finally,
\[
\norm{D_v-D_{v^{[N]}}}_{\ell^{q_1}\to\ell^{q_2}}
=
\left(
\sum_{k>N}v_k^{q_*}
\right)^{\frac{1}{q_*}}
\longrightarrow0.
\]
Hence every bounded \(D_v\) is compact when \(q_2<q_1\).
\end{proof}

\medskip
\noindent
For \(\gamma_0\geq0\), \(\gamma_1,\tau>0\), and \(\sigma \in\mathbb R\), define,
initially on finitely supported sequences,
\[
(\mathcal T x)_\ell
:=
\sum_{k=0}^\infty
\e^{-\gamma_1\abs{\ell-\tau k}^2}
\e^{-\gamma_0k^2}
(1+k)^{\sigma} x_k,
\qquad \ell\in\mathbb N_0.
\]
For \(N\in\mathbb N_0\), let
\[
(\mathcal T_Nx)_\ell
:=
\sum_{k>N}
\e^{-\gamma_1\abs{\ell-\tau k}^2}
\e^{-\gamma_0k^2}
(1+k)^{\sigma} x_k.
\]
Thus \(\mathcal T-\mathcal T_N\) has finite rank.

\begin{proposition}[Drifted Gaussian sequence operators]
\label{prop:drifted-gaussian-operator}
Let \(\gamma_0\geq0\), \(\gamma_1,\tau>0\), \(\sigma\in\mathbb R\), and
\(0<q_1,q_2\leq\infty\).

\begin{enumerate}[label=\textup{(\alph*)}]
\item
If \(\gamma_0>0\), then
\(
\mathcal T:
\ell^{q_1}
\longrightarrow
\ell_{\mathrm{sh}}^{q_2}
\)
is compact. Moreover, 
\[
\norm{\mathcal T_N}_{\ell^{q_1}\to\ell_{\mathrm{sh}}^{q_2}}
\lesssim \e^{- \frac{\gamma_0}{2} N^2},
\qquad N\in\mathbb N_0.
\]

\item
Suppose that \(\gamma_0=0\) and \(q_1\leq q_2\). Then
\(
\mathcal T:
\ell^{q_1}
\longrightarrow
\ell_{\mathrm{sh}}^{q_2}
\)
is bounded if and only if \(\sigma + \frac{1}{q_2}\leq0\), and compact if and only if
\(\sigma + \frac{1}{q_2}<0\). Moreover, whenever \(\sigma+ \frac{1}{q_2}\leq0\),
\[
\norm{\mathcal T_N}_{\ell^{q_1}\to\ell_{\mathrm{sh}}^{q_2}}
\asymp
(1+N)^{\sigma + \frac{1}{q_2}}.
\]

\item
Suppose that \(\gamma_0=0\) and \(q_2<q_1\).
Then boundedness and compactness of
\(
\mathcal T:
\ell^{q_1}
\longrightarrow
\ell_{\mathrm{sh}}^{q_2}
\)
are equivalent, and they hold if and only if
\(
\sigma + \frac{1}{q_2} + \frac1{q_*} < 0.
\)
Under this condition,
\[
\norm{\mathcal T_N}_{\ell^{q_1}\to\ell_{\mathrm{sh}}^{q_2}}
\asymp
(1+N)^{\sigma + \frac{1}{q_2} + \frac1{q_*}}.
\]
\end{enumerate}
The implicit constants depend only on
\(\gamma_0,\gamma_1,\tau,\sigma,q_1\), and \(q_2\).
\end{proposition}

\begin{proof}
Set
\[
v_k
:=
\e^{-\gamma_0k^2}
(1+k)^{\sigma+\frac1{q_2}},
\qquad k\geq0.
\]
We first show that
\begin{equation}
\label{eq:drifted-diagonal-equivalence}
\norm{\mathcal T}_{\ell^{q_1}\to\ell_{\mathrm{sh}}^{q_2}}
\asymp
\norm{D_v}_{\ell^{q_1}\to\ell^{q_2}}.
\end{equation}
Since
\[
1+\ell
\lesssim
(1+k)\bigl(1+\abs{\ell-\tau k}\bigr) \qquad \text{and} \qquad
(1+s)^{\frac{1}{q_2}}\e^{-\frac{\gamma_1}{2}s^2}
\lesssim 1,
\qquad s\geq0,
\]
we have
\[
(1+\ell)^{\frac{1}{q_2}}
\e^{-\gamma_1\abs{\ell-\tau k}^2}
\lesssim
(1+k)^{\frac{1}{q_2}}
\e^{-\frac{\gamma_1}{2}\abs{\ell-\tau k}^2}.
\]
Consequently,
\[
\begin{aligned}
(1+\ell)^{\frac{1}{q_2}}\abs{(\mathcal Tx)_\ell}
&\lesssim
\sum_{k=0}^\infty
\e^{-\frac{\gamma_1}{2}\abs{\ell-\tau k}^2}
v_k\abs{x_k}.
\end{aligned}
\]
Lemma~\ref{lem:drifted-gaussian-matrix}, applied with
\(\gamma= \frac{\gamma_1}{2}\), therefore gives
\[
\norm{\mathcal Tx}_{\ell_{\mathrm{sh}}^{q_2}}
\lesssim
\norm{D_v\abs{x}}_{\ell^{q_2}}.
\]
Consequently,
\begin{equation}
\label{eq:drifted-diagonal-upper}
\norm{\mathcal T}_{\ell^{q_1}\to\ell_{\mathrm{sh}}^{q_2}}
\lesssim
\norm{D_v}_{\ell^{q_1}\to\ell^{q_2}}.
\end{equation}

\smallskip
\noindent
For the reverse estimate, set
\[
m_k:=\lfloor\tau k\rfloor,
\qquad k\geq0.
\]
Then
\[
\abs{m_k-\tau k}<1
\qquad\text{and}\qquad
1+m_k\asymp1+k.
\]
Moreover, if \(m_k=m_{k'}\), then
\[
\tau\abs{k-k'}<1.
\]
Thus the fibers of \(k\mapsto m_k\) have uniformly bounded
cardinality, and \(\mathbb N_0\) can be partitioned into finitely
many sets \(E_1,\ldots,E_J\) such that \(k\mapsto m_k\) is injective
on each \(E_j\).

\smallskip
\noindent
Let \(x\) be a nonnegative sequence supported on \(E_j\). For \(k\in E_j\),
\[
(\mathcal Tx)_{m_k}
\geq
\e^{-\gamma_1}
\e^{-\gamma_0k^2}
(1+k)^\sigma x_k.
\]
Since the coordinates \(m_k\), \(k\in E_j\), are distinct and
\(1+m_k\asymp1+k\), we obtain
\[
\norm{\mathcal Tx}_{\ell_{\mathrm{sh}}^{q_2}}
\geq
\left\|
\left(
(1+m_k)^{\frac1{q_2}}
(\mathcal Tx)_{m_k}
\right)_{k\in E_j}
\right\|_{\ell^{q_2}} \gtrsim
\left\|
\left(
v_kx_k
\right)_{k\in E_j}
\right\|_{\ell^{q_2}}
=
\norm{D_vx}_{\ell^{q_2}},
\]
with the usual supremum interpretation when \(q_2=\infty\).
Applying this estimate to \(\abs{x}\chi_{E_j}\) yields
\[
\norm{D_{v\chi_{E_j}}}_{\ell^{q_1}\to\ell^{q_2}}
\lesssim
\norm{\mathcal T}_{\ell^{q_1}\to\ell_{\mathrm{sh}}^{q_2}},
\qquad 1\leq j\leq J.
\]
By Lemma~\ref{lem:diagonal-sequence-operator},
\[
\norm{D_v}_{\ell^{q_1}\to\ell^{q_2}}
=
\begin{cases}
\displaystyle
\max_{1\leq j\leq J}
\norm{D_{v\chi_{E_j}}}_{\ell^{q_1}\to\ell^{q_2}},
& q_1\leq q_2,\\[4mm]
\displaystyle
\left(
\sum_{j=1}^J
\norm{D_{v\chi_{E_j}}}_{\ell^{q_1}\to\ell^{q_2}}^{q_*}
\right)^{\frac1{q_*}},
& q_2<q_1.
\end{cases}
\]
Since \(J\) is fixed, it follows that
\[
\norm{D_v}_{\ell^{q_1}\to\ell^{q_2}}
\lesssim
\norm{\mathcal T}_{\ell^{q_1}\to\ell_{\mathrm{sh}}^{q_2}}.
\]
Together with \eqref{eq:drifted-diagonal-upper}, this proves
\eqref{eq:drifted-diagonal-equivalence}.

\medskip
\noindent
Applying the same argument to the sets
\[
E_j\cap\{k>N\},
\qquad 1\leq j\leq J,
\]
gives
\begin{equation}
\label{eq:drifted-tail-equivalence}
\norm{\mathcal T_N}_{\ell^{q_1}\to\ell_{\mathrm{sh}}^{q_2}}
\asymp
\norm{D_{v^{(N)}}}_{\ell^{q_1}\to\ell^{q_2}},
\qquad
v_k^{(N)}:=v_k\chi_{\{k>N\}},
\end{equation}
with constants independent of \(N\).

\smallskip
\noindent
Suppose first that \(\gamma_0>0\). The Gaussian factor in \(v_k\)
dominates every polynomial power. Hence
Lemma~\ref{lem:diagonal-sequence-operator} and
\eqref{eq:drifted-tail-equivalence} give
\[
\norm{\mathcal T_N}_{\ell^{q_1}\to\ell_{\mathrm{sh}}^{q_2}}
\lesssim
\e^{- \frac{\gamma_0}{2} N^2}.
\]
 Since
\(\mathcal T-\mathcal T_N\) has finite rank, \(\mathcal T\) is
compact. This proves \textup{(a)}.

\smallskip
\noindent
Suppose next that \(\gamma_0=0\) and \(q_1\leq q_2\). Then
\[
v_k=(1+k)^{\sigma+\frac1{q_2}}.
\]
By Lemma~\ref{lem:diagonal-sequence-operator}\textup{(a)} and
\eqref{eq:drifted-diagonal-equivalence}, \(\mathcal T\) is bounded if
and only if
\(
\sigma+\frac1{q_2}\leq0.
\)
Under this condition, by \eqref{eq:drifted-tail-equivalence},
\[
\norm{\mathcal T_N}_{\ell^{q_1}\to\ell_{\mathrm{sh}}^{q_2}}
\asymp
\sup_{k>N}(1+k)^{\sigma+\frac1{q_2}}\\
\asymp
(1+N)^{\sigma+\frac1{q_2}}.
\]
Thus \(\mathcal T\) is compact when
\(\sigma+\frac{1}{q_2}<0\).

\medskip
\noindent
Conversely, suppose that
\(
\sigma+\frac1{q_2}=0.
\)
Then \(v_k=1\). Since each \(e_k\) is supported on one of the sets
\(E_j\), the upper and reverse estimates used in proving
\eqref{eq:drifted-diagonal-equivalence} give
\[
\norm{\mathcal Te_k}_{\ell_{\mathrm{sh}}^{q_2}}
\asymp
\norm{D_ve_k}_{\ell^{q_2}}
=1,
\qquad k\geq0.
\]
Moreover, for every fixed \(\ell\),
\[
(\mathcal Te_k)_\ell
=
\e^{-\gamma_1\abs{\ell-\tau k}^2}(1+k)^\sigma
\longrightarrow0.
\]
If \(\mathcal T\) were compact, then some subsequence
\((\mathcal Te_{k_j})_j\) would converge in
\(\ell_{\mathrm{sh}}^{q_2}\). Convergence in this space implies
coordinatewise convergence, so its limit would necessarily be zero.
This contradicts
\(
\norm{\mathcal Te_{k_j}}_{\ell_{\mathrm{sh}}^{q_2}}
\asymp1.
\)
Hence \(\mathcal T\) is not compact.

\medskip
\noindent
Finally, suppose that \(\gamma_0 = 0\) and \(q_2<q_1\). By
Lemma~\ref{lem:diagonal-sequence-operator}\textup{(b)} and
\eqref{eq:drifted-diagonal-equivalence}, \(\mathcal T\) is bounded
if and only if
\[
\left(
(1+k)^{\sigma+\frac1{q_2}}
\right)_{k\geq0}
\in\ell^{q_*},
\qquad \text{
or equivalently,} \qquad
q_*
\left(
\sigma+\frac1{q_2}
\right)
<-1.
\]
This is precisely
\[
\sigma+\frac1{q_2}+\frac1{q_*}<0.
\]
Under this condition, \eqref{eq:drifted-tail-equivalence} and an integral comparison give
\[
\norm{\mathcal T_N}_{\ell^{q_1}\to\ell_{\mathrm{sh}}^{q_2}}
\asymp
\left(
\sum_{k>N}
(1+k)^{q_*\left(\sigma+\frac1{q_2}\right)}
\right)^{\frac{1}{q_*}} \asymp
(1+N)^{\sigma+\frac1{q_2}+\frac1{q_*}}
\longrightarrow0.
\]
Thus \(\mathcal T\) is compact whenever it is bounded, proving
\textup{(c)}.
\end{proof}

%================================================
\section{Boundedness of Fock projections}
\label{sec:boundedness}
%================================================

We now prove the boundedness classification in
Theorem~\ref{thm:main-boundedness}. By the pointwise domination
\eqref{eq:pointwise-domination}, boundedness of $P_t^+$ implies
boundedness of $P_t$. It therefore suffices to prove two substantive
implications. First, assuming that $P_t$ is bounded, four testing
families yield the source restrictions, the Gaussian condition, and
the two endpoint inequalities
(Proposition~\ref{prop:boundedness-necessary}). Conversely, under
these parameter conditions, the unified shell estimate
(Lemma~\ref{lem:unified-shell-estimate}) and annular discretization
reduce $P_t^+$ to a drifted Gaussian sequence operator, whose
boundedness gives the converse implication
(Proposition~\ref{prop:boundedness-sufficient}).

\subsection{Necessary conditions for \(P_t\)}
\label{subsec:boundedness-necessity}

The Gaussian-weighted monomial and exponential-packet tests below
are related to the Banach-range arguments in
\cite[Section~3]{LiuShiWei2027}. We include streamlined
self-contained proofs, since the same tests apply throughout the
quasi-Banach target range. The source-exponent obstructions, which do
not arise under the Banach-range assumptions of that paper, are
obtained from thin radial shells and thin angular sectors.

\begin{proposition}
\label{prop:boundedness-necessary}
Suppose that
\(
P_t:
\Lpqa{\alpha}{p_1}{q_1}
\longrightarrow
\Lpqa{\beta}{p_2}{q_2}
\)
is bounded. Then
\[
p_1,q_1\geq1
\qquad\text{and}\qquad
\Delta(t;\alpha,\beta)\leq0.
\]
If \(\Delta(t;\alpha,\beta)=0\), then
\[
q_1\leq q_2
\qquad\text{and}\qquad
d(p_2,q_2)\leq d(p_1,q_1).
\]
\end{proposition}

\begin{proof}
We divide the proof into three steps. 

\medskip
\noindent
\emph{Step 1. Source exponents.} We show that \(p_1 \geq 1 \) and \( q_1 \geq 1\). Since each assertion is
automatic when the corresponding exponent is infinite, it suffices to
consider \(q_1<\infty\) when proving \(q_1\geq1\), and \(p_1<\infty\)
when proving \(p_1\geq1\).

\smallskip
\noindent
For \(0<\varepsilon< 1\), using 
\(f_\varepsilon
:=
\chi_{\{1\leq\abs{w}<1+\varepsilon\}}\), we get
\[
\norm{f_\varepsilon}_{p_1,q_1,\alpha} = \left( q_1 \alpha \int_1^{1 + \varepsilon} e^{-\frac{q_1 \alpha}{2} r^2}r\dd r \right)^{\frac{1}{q_1}}
\lesssim
\varepsilon^{\frac{1}{q_1}},
\]
whereas
\[
(P_tf_\varepsilon)(0)
=
2t\int_1^{1+\varepsilon}\e^{-tr^2}r\,dr
\gtrsim
\varepsilon.
\]
Since \(P_tf_\varepsilon\) is entire, Lemma~\ref{lem:point-evaluation}
and the boundedness of \(P_t\) give
\[
|(P_tf_\varepsilon)(0)| \lesssim \|P_t\|_{\Lpqa{\alpha}{p_1}{q_1}
\longrightarrow
\Lpqa{\beta}{p_2}{q_2}} \|f_{\varepsilon\|_{p_1, q_1, \alpha}}, \qquad \text{hence, } \qquad
\varepsilon
\lesssim
\varepsilon^{\frac{1}{q_1}}.
\]
Letting \(\varepsilon\to0^+\) yields \(q_1\geq1\).

\smallskip
\noindent
Next, for \(0 < \varepsilon < 1\), using \(f_\varepsilon:=\chi_{E_\varepsilon}\) with
\(
E_\varepsilon
:=
\{
re^{i\theta}:0\leq r<1,\ 0\leq\theta<\varepsilon
\}
\),
we obtain
\[
\norm{f_\varepsilon}_{p_1,q_1,\alpha} \ \lesssim \ 
\varepsilon^{\frac{1}{p_1}},
\]
while
\[
(P_tf_\varepsilon)(0)
=
\frac{t}{\pi}
\int_0^1\int_0^\varepsilon
\e^{-tr^2}r\,d\theta\,dr
\ \gtrsim \
\varepsilon.
\]
The preceding point-evaluation argument gives
\[
\varepsilon
\lesssim
\varepsilon^{\frac{1}{p_1}},
\]
and letting \(\varepsilon\to0^+\) yields \(p_1\geq1\).

\medskip
\noindent
\emph{Step 2. Gaussian obstructions.}
We show that \(\Delta(t; \alpha, \beta) \leq 0\).
Let
\[
\gamma<\min\left\{t,\frac{\alpha}{2}\right\},
\qquad
\qquad \text{and} \qquad
f_{\gamma,n}(w)
:=
w^n\e^{\gamma\abs{w}^2},
\qquad n\in\mathbb N.
\]
By \eqref{eq:monomial-norm},
\[
\norm{f_{\gamma,n}}_{p_1,q_1,\alpha}
\asymp
\left(\frac{n!}{(\alpha - 2\gamma)^n}\right)^{\frac{1}{2}}
(1+n)^{\frac1{2q_1}-\frac14}.
\]
Lemma~\ref{lem:fock-test-formulas} gives
\[
P_tf_{\gamma,n}(z)
=
\left(\frac{t}{t-\gamma}\right)^{n+1}z^n,
\]
and hence, again by \eqref{eq:monomial-norm},
\[
\norm{P_tf_{\gamma,n}}_{p_2,q_2,\beta}
\asymp
\left(\frac{t}{t-\gamma}\right)^{n+1}
\left(\frac{n!}{\beta^n}\right)^{\frac{1}{2}}
(1+n)^{\frac1{2q_2}-\frac14}.
\]
The boundedness of \(P_t\) therefore implies
\begin{equation}
\label{eq:monomial-base-obstruction}
\left(\frac{t}{t-\gamma}\right)^{n+1} \left(\frac{\alpha - 2\gamma}{\beta}\right)^{\frac{n}{2}}
(1+n)^{\frac1{2q_2}-\frac1{2q_1}}
\lesssim 1.
\end{equation}
Taking \(n\)-th roots in \eqref{eq:monomial-base-obstruction} and letting \(n \to \infty\), we obtain
\[
\frac{t^2 (\alpha - 2 \gamma)}{(t-\gamma)^2\beta} 
\leq 1.
\]
If \(\alpha\geq2t\), letting \(\gamma\uparrow t\) in the last inequality gives a contradiction. Thus
\(\alpha<2t\), and we may take \(\gamma=\alpha-t\). Then 
\(
\alpha-2\gamma=t-\gamma=2t-\alpha,
\)
hence,
\[
\frac{t^2}{\beta(2t-\alpha)}\leq1,
\qquad \text{
or equivalently,} \qquad
\Delta(t;\alpha,\beta)
=
t^2-\beta(2t-\alpha)
\leq0.
\]

\medskip
\noindent
\emph{Step 3. Obstructions at the Gaussian boundary.} Assume that \(\Delta(t;\alpha,\beta)=0\). Then \(
t^2=\beta(2t-\alpha)
\). We show that
\[
q_1\leq q_2
\qquad\text{and}\qquad
d(p_2,q_2)\leq d(p_1,q_1).
\]
Taking \(\gamma=\alpha-t\) in
\eqref{eq:monomial-base-obstruction}, the factor that is
exponential in \(n\) is equal to one. Hence
\[
(1+n)^{\frac1{2q_2}-\frac1{2q_1}}
\lesssim1,
\qquad n\geq1,
\]
and therefore \(q_1\leq q_2\).

\smallskip
\noindent
For \(w\in\C\), define
\[
f_w(z)
:=
\e^{z\overline{w}}\e^{(\alpha-t)\abs{z}^2}.
\]
By \eqref{eq:exponential-kernel-norm},
\[
\norm{f_w}_{p_1,q_1,\alpha}
\asymp
\exp\left(
\frac{\abs{w}^2}{2(2t-\alpha)}
\right)
(1+\abs{w})^{d(p_1,q_1)}.
\]
Moreover, Lemma~\ref{lem:fock-test-formulas} gives
\[
P_tf_w(z)
=
\frac{t}{2t-\alpha}
\exp\left(
\frac{t}{2t-\alpha}z\overline{w}
\right),
\]
and hence, by \eqref{eq:exponential-kernel-norm} again,
\[
\norm{P_tf_w}_{p_2,q_2,\beta}
\asymp
\exp\left(
\frac{t^2\abs{w}^2}
     {2\beta(2t-\alpha)^2}
\right)
(1+\abs{w})^{d(p_2,q_2)}.
\]
Since \(t^2=\beta(2t-\alpha)\), the Gaussian factors coincide.
Boundedness therefore gives
\[
(1+\abs{w})^{d(p_2,q_2)}
\lesssim
(1+\abs{w})^{d(p_1,q_1)}.
\]
Letting \(\abs{w}\to\infty\), we obtain
\(
d(p_2,q_2)\leq d(p_1,q_1).
\)
\end{proof}

\subsection{A unified shell estimate for \(P_t^+\)} \label{subsec:unified-shell-estimate}

We next derive a shell estimate that reduces the maximal Fock
projection to the drifted Gaussian sequence operators studied in
Subsection~\ref{subsec:drifted-gaussian-operators}. In the Banach
mixed-norm range, sufficiency was proved in
\cite[Section~4]{LiuShiWei2027} using continuous radial kernel
estimates. Our shell argument is self-contained, applies also to
quasi-Banach targets, and provides the tail estimates needed in the
compactness argument.

\medskip
\noindent
For \(1\leq p_1, p_2\leq\infty\), define the angular convolution exponent
\(\rho(p_1, p_2)\) by
\[
\frac1{\rho(p_1, p_2)}
:=
\begin{cases}
1, & p_2 \leq p_1,\\[2mm]
1+\dfrac{1}{p_2}-\dfrac{1}{p_1}, & p_1<p_2.
\end{cases}
\]
Throughout this subsection, set
\[
p_2^\sharp:=\max\{1,p_2\},
\qquad
\rho_{12}:=\rho(p_1,p_2^\sharp),
\qquad
\sigma:=\frac1{q_1'}-\frac1{\rho_{12}},
\]
and
\[
\tau:=\frac{t}{\beta},
\qquad
\eta:=-\frac{\Delta(t;\alpha,\beta)}{2\beta}.
\]
Thus \(\eta>0\) in the strict Gaussian region, whereas
\(\eta=0\) at the Gaussian boundary.

\medskip
\noindent
For \(k\in\mathbb N_0\), define
\[
X_{p,q,\alpha}(f;k)
:=
\begin{cases}
\displaystyle
\left(
q\alpha\int_k^{k+1}
M_p(f,s)^q\e^{-\frac{q\alpha}{2}s^2}s\,ds
\right)^{\frac{1}{q}},
& 0<q<\infty,\\[4mm]
B_{p,\alpha}(f;k),
& q=\infty.
\end{cases}
\]
Then
\begin{equation}
\label{eq:source-shell-sequence-norm}
\norm{\bigl(X_{p,q,\alpha}(f;k)\bigr)_{k\geq0}}_{\ell^q}
=
\norm{f}_{p,q,\alpha}.
\end{equation}

\begin{lemma}[Unified shell estimate]
\label{lem:unified-shell-estimate}
Assume that
\[
1\leq p_1,q_1\leq\infty,
\qquad
0<p_2\leq\infty,
\qquad
\Delta(t;\alpha,\beta)\leq0.
\]
Then, for every nonnegative measurable function \(g\),
\begin{equation}
\label{eq:unified-shell-estimate}
B_{p_2,\beta}(P_t^+g;\ell)
\lesssim
\sum_{k=0}^\infty
\e^{-\frac{\beta}{4}\abs{\ell-\tau k}^2}
\e^{-\eta k^2}
(1+k)^\sigma
X_{p_1,q_1,\alpha}(g;k),
\qquad \ell\geq0,
\end{equation}
where both sides are allowed to take the value \(+\infty\).
 The implicit constant depends only on
\(t,\alpha,\beta,p_1,q_1\), and \(p_2\).
\end{lemma}

\begin{proof}
For \(z=re^{i\theta}\) and \(w=se^{i\phi}\), set
\[
E_a(\theta):=\e^{a\cos\theta},
\qquad
g_s(\phi):=g(se^{i\phi}).
\]
Since convolution on \(\T\) is taken with respect to normalized
Lebesgue measure, Tonelli's theorem gives
\[
P_t^+g(re^{i\theta})
=
2t\int_0^\infty
(E_{trs}*g_s)(\theta)
\e^{-ts^2}s\,ds.
\]
The definition of \(\rho_{12}\) ensures that
\begin{equation}
\label{eq:angular-young-estimate}
\norm{E_a*h}_{L^{p_2^\sharp}(\T)}
\leq
\norm{E_a}_{L^{\rho_{12}}(\T)}
\norm{h}_{L^{p_1}(\T)}.
\end{equation}
Indeed, if \(p_1<p_2^\sharp\), then
\[
1+\frac1{p_2^\sharp}
=
\frac1{p_1}+\frac1{\rho_{12}},
\]
and \eqref{eq:angular-young-estimate} follows from Young's convolution
inequality. If \(p_2^\sharp\leq p_1\), then \(\rho_{12}=1\), and the
same conclusion follows from Young's inequality in \(L^{p_1}(\T)\)
and the embedding
\[
L^{p_1}(\T)\hookrightarrow L^{p_2^\sharp}(\T).
\]
We refer to \cite[Theorem~1.2.12]{Grafakos2014} for Young's
convolution inequality.

\smallskip
\noindent
Since \(p_2^\sharp\geq1\), by Minkowski's integral inequality
\cite[Theorem~6.19(a)]{Folland1999}, with the usual supremum
argument when \(p_2^\sharp=\infty\), and the preceding convolution
estimate \eqref{eq:angular-young-estimate}, we obtain
\[
\begin{aligned}
M_{p_2^\sharp}(P_t^+g,r)
&\leq
2t\int_0^\infty
\norm{E_{trs}*g_s}_{L^{p_2^\sharp}(\T)}
\e^{-ts^2}s\,ds\\
&\leq
2t\int_0^\infty
\norm{E_{trs}}_{L^{\rho_{12}}(\T)}
M_{p_1}(g,s)\e^{-ts^2}s\,ds.
\end{aligned}
\]

\medskip
\noindent
Moreover, for \(1\leq\rho<\infty\), the integral representation of the modified
Bessel function \(I_0\) gives
\[
\norm{E_a}_{L^\rho(\T)}^\rho
=
\frac1{2\pi}\int_{-\pi}^{\pi}
\e^{\rho a\cos\theta}\,d\theta
=
I_0(\rho a).
\]
Using the standard estimate for \(I_0\), see, for example, \cite[formulas~9.6.19 and~9.7.1]{AbramowitzStegun1964},
we obtain
\[
\norm{E_a}_{L^\rho(\T)}
\lesssim
\e^a(1+a)^{-\frac1{2\rho}}.
\]
For \(\rho=\infty\), the same estimate is immediate from
\(
\norm{E_a}_{L^\infty(\T)}=\e^a.
\) 
Since \(\rho_{12}\) is fixed,
\[
\norm{E_a}_{L^{\rho_{12}}(\T)}
\lesssim
\e^a(1+a)^{-\frac1{2\rho_{12}}}.
\]
Consequently,
\[
M_{p_2^\sharp}(P_t^+g,r)
\lesssim
\int_0^\infty
\e^{trs}(1+rs)^{-\frac1{2\rho_{12}}}
M_{p_1}(g,s)\e^{-ts^2}s\,ds.
\]
Since \(p_2\leq p_2^\sharp\) and \(\T\) has normalized measure,
\[
M_{p_2}(P_t^+g,r)
\leq
M_{p_2^\sharp}(P_t^+g,r).
\]
Hence, 
\[
M_{p_2}(P_t^+g,r)
\lesssim
\int_0^\infty
\e^{trs}(1+rs)^{-\frac1{2\rho_{12}}}
M_{p_1}(g,s)\e^{-ts^2}s\,ds.
\]
\medskip
\noindent
The Gaussian identity
\[
trs-ts^2-\frac{\beta r^2}{2}
=
-\frac{\beta}{2}(r-\tau s)^2
-\frac{\alpha s^2}{2}
-\eta s^2
\]
therefore yields
\[
M_{p_2}(P_t^+g,r)\e^{-\frac{\beta}{2}r^2} \lesssim
\int_0^\infty
\e^{-\frac{\beta}{2}(r-\tau s)^2}
\e^{-\eta s^2}
(1+rs)^{-\frac1{2\rho_{12}}}
M_{p_1}(g,s)\e^{-\frac{\alpha}{2}s^2}s\,ds.
\]
Let \(r\in[\ell,\ell+1)\) and \(s\in[k,k+1)\). The elementary
relations
\[
\abs{r-\tau s}
\geq
\abs{\ell-\tau k}- (1 + \tau)
\qquad
\text{and} \qquad
1+k
\lesssim
(1+rs)^{\frac{1}{2}}\bigl(1+\abs{r-\tau s}\bigr)
\]
imply, after slightly decreasing the Gaussian constant, that
\[
\e^{-\frac{\beta}{2}(r-\tau s)^2}
\e^{-\eta s^2} (1+rs)^{-\frac1{2\rho_{12}}}
\lesssim
\e^{-\frac{\beta}{4}\abs{\ell-\tau k}^2}
\e^{-\eta k^2}(1+k)^{-\frac1{\rho_{12}}}.
\]
Finally, Hölder's inequality on \([k,k+1)\), with the usual interpretation when
\(q_1=\infty\), gives
\[
\int_k^{k+1}M_{p_1}(g,s)\e^{-\frac{\alpha}{2}s^2}s\,ds
\lesssim
(1+k)^{\frac1{q_1'}}
X_{p_1,q_1,\alpha}(g;k).
\]
Combining the preceding estimates and taking the supremum over
\(r\in[\ell,\ell+1)\) proves
\eqref{eq:unified-shell-estimate}.
\end{proof}

\subsection{Sufficient conditions for \(P_t^+\)}
\label{subsec:boundedness-sufficiency}

We now combine Lemma~\ref{lem:unified-shell-estimate} with
Proposition~\ref{prop:drifted-gaussian-operator} to prove the converse
implication.

\begin{proposition}
\label{prop:boundedness-sufficient}
Let
\(
1\leq p_1,q_1\leq\infty\) and \(
0<p_2,q_2\leq\infty\).
Suppose that either
\[
\Delta(t;\alpha,\beta)<0,
\]
or
\[
\Delta(t;\alpha,\beta)=0,
\qquad
q_1\leq q_2,
\qquad
d(p_2,q_2)\leq d(p_1,q_1).
\]
Then
\(
P_t^+:
\Lpqa{\alpha}{p_1}{q_1}
\longrightarrow
\Lpqa{\beta}{p_2}{q_2}
\)
is bounded.
\end{proposition}

\begin{proof}
Retain the notation from Subsection~\ref{subsec:unified-shell-estimate}.
We first prove absolute convergence of the defining integral. Since
\(\Delta(t;\alpha,\beta)\leq0\),
\(
\beta(2t-\alpha)\geq t^2>0,
\)
and hence
\[
\mu:=t-\frac{\alpha}{2}>0.
\]
Fix \(R>0\). For \(\abs z\leq R\), since \(p_1\geq1\),
\[
\begin{aligned}
\int_{\C}
\abs{f(w)}
\e^{t\operatorname{Re}(z\bar w)}
\e^{-t\abs{w}^2}\,dA(w) 
\lesssim
\int_0^\infty
M_{p_1}(f,s)
\e^{-\frac{\alpha}{2}s^2}
\e^{tRs-\mu s^2}s\,ds.
\end{aligned}
\]
Since \(\mu >0\), the function
\(
s\longmapsto\e^{tRs- \mu s^2}
\)
belongs to the space
\(
L^{q_1'}\bigl((0,\infty),s\,ds\bigr),
\)
with the usual endpoint interpretations when \(q_1=1\) or
\(q_1=\infty\). Hence Hölder's inequality gives
\[
\int_{\C}
\abs{f(w)}
\e^{t\operatorname{Re}(z\bar w)}
\e^{-t\abs{w}^2}\,dA(w)
\lesssim
\norm{f}_{p_1,q_1,\alpha},
\]
where the implicit constant depends on \(\alpha, t, q_1\), and \(R\). Thus the defining integrals for both \(P_t^+f\) and \(P_tf\)
converge absolutely and locally uniformly in \(z\). In particular,
\(P_t^+f\) is continuous and \(P_tf\) is entire.

\medskip
\noindent
We next prove the mixed-norm estimates. Set
\[
g:=\abs f,
\qquad
Y_k:=X_{p_1,q_1,\alpha}(g;k),
\qquad k\geq0.
\]
By \eqref{eq:source-shell-sequence-norm},
\[
\norm{Y}_{\ell^{q_1}}
=
\norm{f}_{p_1,q_1,\alpha}.
\]
Lemma~\ref{lem:unified-shell-estimate} and
Proposition~\ref{prop:annular-discretization}, with the defining
supremum used when \(q_2=\infty\), give
\begin{equation}
\label{eq:positive-projection-sequence-bound}
\norm{P_t^+g}_{p_2,q_2,\beta}
\lesssim
\norm{\mathcal TY}_{\ell_{\mathrm{sh}}^{q_2}},
\end{equation}
where
\[
(\mathcal TY)_\ell
:=
\sum_{k=0}^\infty
\e^{-\frac{\beta}{4}\abs{\ell-\tau k}^2}
\e^{-\eta k^2}
(1+k)^\sigma Y_k.
\]

\smallskip
\noindent
\emph{The strict Gaussian region.}
If \(\Delta(t;\alpha,\beta)<0\), then \(\eta>0\).
Proposition~\ref{prop:drifted-gaussian-operator}\textup{(a)} gives
\(
\norm{\mathcal TY}_{\ell_{\mathrm{sh}}^{q_2}}
\lesssim
\norm{Y}_{\ell^{q_1}}.
\)
Together with \eqref{eq:positive-projection-sequence-bound}, this
yields
\[
\norm{P_t^+g}_{p_2,q_2,\beta}
\lesssim
\norm{f}_{p_1,q_1,\alpha}.
\]

\smallskip
\noindent
\emph{The Gaussian boundary.}
Suppose that \(\Delta(t;\alpha,\beta)=0\). Then \(\eta=0\).
If \(p_2^\sharp\leq p_1\), then \(\rho_{12}=1\), and
\[
\sigma+\frac1{q_2}
=
\frac1{q_2}-\frac1{q_1}
\leq0.
\]
If \(p_1<p_2^\sharp\), then necessarily
\(p_2^\sharp=p_2\), and
\[
\sigma+\frac1{q_2}
=
\frac1{q_1'}
-\left(
1+\frac1{p_2}-\frac1{p_1}
\right)
+\frac1{q_2}=
d(p_2,q_2)-d(p_1,q_1)
\leq0.
\]
Since \(q_1\leq q_2\),
Proposition~\ref{prop:drifted-gaussian-operator}\textup{(b)} gives
\(
\norm{\mathcal TY}_{\ell_{\mathrm{sh}}^{q_2}}
\lesssim
\norm{Y}_{\ell^{q_1}}.
\)
Hence
\[
\norm{P_t^+g}_{p_2,q_2,\beta}
\lesssim
\norm{f}_{p_1,q_1,\alpha}.
\]
Finally, since 
\(
\abs{P_t^+f}
\leq
P_t^+\abs f\), we obtain
\[
\norm{P_t^+f}_{p_2,q_2,\beta}
\lesssim
\norm{f}_{p_1,q_1,\alpha},
\]
which completes the proof.
\end{proof}

\medskip
\noindent
Theorem~\ref{thm:main-boundedness} immediately yields the following
special cases.

\begin{corollary}
\label{cor:equal-gaussian-parameters}
Let \(t, \alpha>0\) and
\(0<p_1,q_1,p_2,q_2\leq\infty\). Then
\(
P_t,P_t^+:
\Lpqa{\alpha}{p_1}{q_1}
\longrightarrow
\Lpqa{\alpha}{p_2}{q_2}
\)
are bounded if and only if
\[
t=\alpha,
\qquad
p_1,q_1\geq1,
\qquad
q_1\leq q_2,
\qquad \text{and} \qquad
d(p_2,q_2)\leq d(p_1,q_1).
\]
\end{corollary}

\begin{corollary}
\label{cor:equal-mixed-exponents}
Let \(t, \alpha,\beta>0\) and \(0<p,q\leq\infty\). Then
\(
P_t,P_t^+:
\Lpqa{\alpha}{p}{q}
\longrightarrow
\Lpqa{\beta}{p}{q}
\)
are bounded if and only if
\[
p,q\geq1
\qquad\text{and}\qquad
\Delta(t;\alpha,\beta)\leq0.
\]
\end{corollary}

%================================================
\section{Compactness of Fock projections}
\label{sec:compactness}
%================================================
We now prove the compactness classification in
Theorem~\ref{thm:main-compactness}. The argument has three parts.
First, the normalized Gaussian-weighted monomial and exponential
families from the boundedness proof show that compactness of $P_t$
forces the two boundary inequalities to be strict
(Proposition~\ref{prop:compactness-necessary}). Second, we prove the
structural implication that compactness of $P_t^+$ implies compactness
of $P_t$
(Proposition~\ref{prop:maximal-implies-fock-compact}). Finally, under
the strict parameter conditions, we approximate $P_t^+$ by compact
source truncations, with the source tails controlled by
Proposition~\ref{prop:drifted-gaussian-operator}
(Proposition~\ref{prop:compactness-sufficient}).

\medskip
\noindent
As shown in the proof of
Proposition~\ref{prop:boundedness-sufficient}, under the equivalent
conditions in Theorem~\ref{thm:main-boundedness}, the integrals
defining \(P_tf\) and \(P_t^+f\) converge absolutely and locally
uniformly in \(z\in\mathbb C\) for every
\(f\in\Lpqa{\alpha}{p_1}{q_1}\). Consequently, \(P_tf\) is entire and
\(P_t^+f\) is continuous. Thus
\[
P_t:
\Lpqa{\alpha}{p_1}{q_1}
\longrightarrow
\Fpqa{\beta}{p_2}{q_2},
\qquad
P_t^+:
\Lpqa{\alpha}{p_1}{q_1}
\longrightarrow
C(\mathbb C)\cap\Lpqa{\beta}{p_2}{q_2}
\]
are bounded.
Since \(\Fpqa{\beta}{p_2}{q_2}\) is a closed subspace of
\(\Lpqa{\beta}{p_2}{q_2}\), compactness of \(P_t\) as an operator into
\(\Lpqa{\beta}{p_2}{q_2}\) is equivalent to compactness as an operator
into \(\Fpqa{\beta}{p_2}{q_2}\).

\subsection{Necessary conditions for \(P_t\)}
\label{subsec:compactness-necessity}

Compactness implies boundedness. Hence the conditions in
Proposition~\ref{prop:boundedness-necessary} are available throughout
this subsection. It remains to show that both boundary inequalities
must be strict.

\begin{proposition}
\label{prop:compactness-necessary}
Suppose that
\(
P_t:
\Lpqa{\alpha}{p_1}{q_1}
\longrightarrow
\Fpqa{\beta}{p_2}{q_2}
\)
is compact. Then
\(
p_1,q_1\geq1
\)
and either
\[
\Delta(t;\alpha,\beta)<0,
\]
or
\[
\Delta(t;\alpha,\beta)=0,
\qquad
q_1<q_2,
\qquad
d(p_2,q_2)<d(p_1,q_1).
\]
\end{proposition}

\begin{proof}
By Proposition~\ref{prop:boundedness-necessary},
\(
p_1,q_1\geq1\) and \(
\Delta(t;\alpha,\beta)\leq0\).
Suppose that
\(
\Delta(t;\alpha,\beta)=0.
\)
Then
\[
2t-\alpha=\frac{t^2}{\beta}>0.
\]
We use the same Gaussian-weighted monomial and exponential
families as in the proof of Proposition~\ref{prop:boundedness-necessary}. The corresponding calculations will therefore be recalled only briefly.

\medskip
\noindent
\emph{Step 1. Strict ordering of the outer exponents.} We show that \(q_1 < q_2\). For \(n\in\mathbb N\), set
\[
f_n(w):=w^n\e^{(\alpha-t)\abs{w}^2},
\qquad
g_n:=\frac{f_n}{\norm{f_n}_{p_1,q_1,\alpha}}.
\]
The monomial calculation from the proof of
Proposition~\ref{prop:boundedness-necessary} gives
\[
\norm{f_n}_{p_1,q_1,\alpha}
\asymp
\left(\frac{n!}{(2t - \alpha)^n}\right)^{\frac{1}{2}}
(1+n)^{\frac1{2q_1}-\frac14},
\qquad 
P_tf_n(z)
=
\left(\frac{t}{2t - \alpha}\right)^{n+1}z^n
\]
and
\[
\norm{P_tf_{n}}_{p_2,q_2,\beta}
\asymp
\left(\frac{t}{2t-\alpha}\right)^{n+1}
\left(\frac{n!}{\beta^n}\right)^{\frac{1}{2}}
(1+n)^{\frac1{2q_2}-\frac14}.
\]
Since \(t^2=\beta (2t - \alpha)\), it follows that
\begin{equation}\label{eq:norm-ptg}
\norm{P_t g_n}_{p_2,q_2,\beta}
\asymp
(1+n)^{\frac1{2q_2}-\frac1{2q_1}}.
\end{equation}
On the other hand, for every \(R>0\), Stirling's formula gives
\[
\sup_{\abs z\leq R}\abs{P_tg_n(z)}
\lesssim 
\frac{C_R^n (1+n)^{\frac{1}{4} - \frac{1}{2q_1}}}{\sqrt{n!}}
\longrightarrow0, \qquad \text{where} \qquad C_R := \frac{tR}{(2t - \alpha)^{\frac{1}{2}}}.
\]
Thus \(P_tg_n\to0\) locally uniformly on \(\C\). Since
\(\norm{g_n}_{p_1,q_1,\alpha}=1\), compactness of \(P_t\) and
Corollary~\ref{cor:ambient-operator-compactness}\textup{(b)} imply
\[
\norm{P_tg_n}_{p_2,q_2,\beta}\longrightarrow0.
\]
Together with \eqref{eq:norm-ptg}, this gives \(q_1<q_2\).

\medskip
\noindent
\emph{Step 2. Strict mixed-norm defect.} We show that \(d(p_2, q_2) < d(p_1, q_1)\).
For \(w \in\C\), define
\[
f_w(z)
:=
\e^{z\overline{w}}\e^{(\alpha-t)\abs{z}^2},
\qquad
g_w
:=
\frac{f_w}{\norm{f_w}_{p_1,q_1,\alpha}}.
\]
Repeating the calculation for the Gaussian-weighted exponential
packets from the proof of Proposition~\ref{prop:boundedness-necessary}, we obtain
\[
\norm{f_w}_{p_1,q_1,\alpha}
\asymp
\exp\left(\frac{\abs{w}^2}{2(2t - \alpha)}\right)
(1+\abs{w})^{d(p_1,q_1)},
\qquad
P_tf_w(z)
=
\frac{t}{2t - \alpha}
\exp\left(\frac{t}{2t - \alpha}z \overline{w}\right)
\]
and
\[
\norm{P_tf_w}_{p_2,q_2,\beta}
\asymp
\exp\left(
\frac{t^2\abs{w}^2}
     {2\beta(2t-\alpha)^2}
\right)
(1+\abs{w})^{d(p_2,q_2)}.
\]
Since \(t^2=\beta(2t-\alpha)\), the Gaussian factors coincide. Hence,
\begin{equation}\label{eq:norm-ptg-packets}
\norm{P_t g_w}_{p_2,q_2,\beta}
\asymp
(1+\abs{w})^{d(p_2,q_2)-d(p_1,q_1)}.
\end{equation}
On the other hand, for every \(R>0\),
\[
\sup_{\abs z\leq R}\abs{P_t g_w(z)}
\lesssim
\exp\left(
-\frac{\abs{w}^2}{2(2t - \alpha)}
+\frac{tR\abs{w}}{2t - \alpha}
\right)
(1+\abs{w})^{-d(p_1,q_1)} \longrightarrow0
\]
as \(\abs{w}\to\infty\). Applying compactness and
Corollary~\ref{cor:ambient-operator-compactness}\textup{(b)} along
any sequence \(\abs{w_n}\to\infty\), we obtain
\[
\norm{P_tg_{w_n}}_{p_2,q_2,\beta}\longrightarrow0.
\]
It follows from \eqref{eq:norm-ptg-packets} that
\[
d(p_2,q_2)<d(p_1,q_1).
\]
\end{proof}

\subsection{Source truncations and compactness transfer}
\label{subsec:compactness-transfer}
%------------------------------------------------
We first isolate the contribution of a bounded source region. The
corresponding source truncations are compact, and they will be used to
transfer compactness from $P_t^+$ to $P_t$ by showing that the source
tails vanish in operator quasi-norm.

\medskip
\noindent
For $N\in\mathbb N$, define
\[
Q_Nf:=f\chi_{\{|w|\geq N\}},
\]
and the source truncations
\[
P_{t,N}f(z)
:=
\frac{t}{\pi}
\int_{\abs w<N}
f(w)\e^{tz\overline w}\e^{-t\abs w^2}\,dA(w),
\]
and
\[
P_{t,N}^+f(z)
:=
\frac{t}{\pi}
\int_{\abs w<N}
f(w)\e^{t\operatorname{Re}(z\overline w)}
\e^{-t\abs w^2}\,dA(w).
\]
Then
\[
P_t=P_{t,N}+P_tQ_N,
\qquad
P_t^+=P_{t,N}^++P_t^+Q_N.
\]

\begin{lemma}
\label{lem:compact-source-truncations}
Let
\(
1\leq p_1,q_1\leq\infty\),
\(
0<p_2,q_2\leq\infty\).
Then, for every $N\in\mathbb N$, the operators
\(
P_{t,N}:
\Lpqa{\alpha}{p_1}{q_1}
\longrightarrow
\Fpqa{\beta}{p_2}{q_2}
\)
and
\(
P_{t,N}^+:
\Lpqa{\alpha}{p_1}{q_1}
\longrightarrow
\Lpqa{\beta}{p_2}{q_2}
\)
are compact.
\end{lemma}

\begin{proof}
Since $p_1,q_1\geq1$, H\"older's inequality in the angular and radial
variables gives
\begin{equation}
\label{eq:local-L1-truncation}
\int_{\abs w<N}\abs{f(w)}\,dA(w)
\leq C_N\norm{f}_{p_1,q_1,\alpha}.
\end{equation}
Consequently, after increasing $C_N$ if necessary,
\[
P_{t,N}^+\abs f(z)
\leq
C_N\e^{tN\abs z}\norm{f}_{p_1,q_1,\alpha}.
\]
The same pointwise domination as in
\eqref{eq:pointwise-domination} gives
\[
\abs{P_{t,N}f}
\leq P_{t,N}^+\abs f
\qquad \text{and} \qquad
\abs{P_{t,N}^+f}
\leq P_{t,N}^+\abs f.
\]
Thus the defining integrals converge absolutely and locally uniformly
in $z$. In particular, $P_{t,N}f$ is entire and $P_{t,N}^+f$ is
continuous.
Moreover,
\[
M_{p_2}\bigl(P_{t,N}^+\abs f,r\bigr)
\e^{-\frac{\beta}{2}r^2}
\leq
C_N\e^{tNr-\frac{\beta}{2}r^2}
\norm{f}_{p_1,q_1,\alpha}.
\]
It follows that both operators in the statement are bounded.

\smallskip
\noindent
Next we verify relative compactness with respect to locally
uniform convergence for the maximal truncation. Fix $R>0$. If $\abs z,\abs\zeta\leq R$
and $\abs w<N$, then
\[
\left|
\e^{t\operatorname{Re}(z\overline w)}
-
\e^{t\operatorname{Re}(\zeta\overline w)}
\right|
\leq
t\abs{z-\zeta}\abs w\,\e^{tNR}  \leq
tN\e^{tNR}\abs{z-\zeta}.
\]
Therefore, by \eqref{eq:local-L1-truncation},
\[
\abs{P_{t,N}^+f(z)-P_{t,N}^+f(\zeta)}
\leq
C_{N,R}\abs{z-\zeta}
\norm{f}_{p_1,q_1,\alpha}.
\]
Thus the image under $P_{t,N}^+$ of the source unit ball is locally
uniformly bounded and locally equicontinuous. The Arzel\`a--Ascoli
theorem and a diagonal argument show that it is relatively compact
with respect to locally uniform convergence.

\smallskip
\noindent
Finally, the pointwise domination above gives
\[
\sup_{\norm{f}_{p_1,q_1,\alpha}\leq1}
\norm{P_{t,N}f}_{p_2,q_2,\beta;[R,\infty)}
\leq
\sup_{\norm{f}_{p_1,q_1,\alpha}\leq1}
\norm{P_{t,N}^+\abs f}_{p_2,q_2,\beta;[R,\infty)}
\]
and
\[
\sup_{\norm{f}_{p_1,q_1,\alpha}\leq1}
\norm{P_{t,N}^+f}_{p_2,q_2,\beta;[R,\infty)}
\leq
\sup_{\norm{f}_{p_1,q_1,\alpha}\leq1}
\norm{P_{t,N}^+\abs f}_{p_2,q_2,\beta;[R,\infty)}.
\]
Moreover,
\[
\sup_{\norm{f}_{p_1,q_1,\alpha}\leq1}
\norm{P_{t,N}^+\abs f}_{p_2,q_2,\beta;[R,\infty)}
\leq
C_N
\begin{cases}
\displaystyle
\left(
q_2\beta\int_R^\infty
\e^{q_2tNr-\frac{q_2\beta}{2}r^2}r\,dr
\right)^{\frac1{q_2}},
& 0<q_2<\infty,\\[4mm]
\displaystyle
\sup_{r\geq R}
\e^{tNr-\frac{\beta}{2}r^2},
& q_2=\infty.
\end{cases}
\]
Both tail bounds tend to zero as $R\to\infty$. Hence Corollary~\ref{cor:ambient-operator-compactness}\textup{(a)}, together with its analytic
specialization, shows that $P_{t,N}$ is compact, while the general
case shows that $P_{t,N}^+$ is compact.
\end{proof}

\begin{proposition}[Compactness transfer]
\label{prop:maximal-implies-fock-compact}
Suppose that
\(
P_t^+:
\Lpqa{\alpha}{p_1}{q_1}
\longrightarrow
\Lpqa{\beta}{p_2}{q_2}
\)
is compact. Then
\(
P_t:
\Lpqa{\alpha}{p_1}{q_1}
\longrightarrow
\Fpqa{\beta}{p_2}{q_2}
\)
is compact.
\end{proposition}

\begin{proof}
Compactness implies boundedness. Hence
Theorem~\ref{thm:main-boundedness} gives
\[
p_1,q_1\geq1
\qquad\text{and}\qquad
\Delta(t;\alpha,\beta)\leq0.
\]
In particular,
\[
2t-\alpha\geq\frac{t^2}{\beta}>0.
\]

Since
\(
P_t-P_{t,N}=P_tQ_N,
\)
the pointwise domination \eqref{eq:pointwise-domination} gives
\[
\abs{P_tQ_Nf(z)}
\leq
P_t^+Q_N(\abs f)(z).
\]
Therefore,
\[
\norm{(P_t-P_{t,N})f}_{p_2,q_2,\beta}
\leq
\norm{P_t^+Q_N(\abs f)}_{p_2,q_2,\beta}
\leq
\norm{P_t^+Q_N}
\norm{f}_{p_1,q_1,\alpha},
\]
and hence
\[
\norm{P_t-P_{t,N}}_{
\Lpqa{\alpha}{p_1}{q_1}
\to
\Fpqa{\beta}{p_2}{q_2}}
\leq
\norm{P_t^+Q_N}_{
\Lpqa{\alpha}{p_1}{q_1}
\to
\Lpqa{\beta}{p_2}{q_2}}.
\]
By Lemma~\ref{lem:compact-source-truncations}, each
\(
P_{t,N}:
\Lpqa{\alpha}{p_1}{q_1}
\longrightarrow
\Fpqa{\beta}{p_2}{q_2}
\)
is compact. It therefore remains to prove that
\[
\norm{P_t^+Q_N}_{
\Lpqa{\alpha}{p_1}{q_1}
\to
\Lpqa{\beta}{p_2}{q_2}}
\longrightarrow0.
\]
Indeed, this would show that $P_{t,N}\to P_t$ in operator quasi-norm. Since $\Fpqa{\beta}{p_2}{q_2}$ is complete, $P_t$
would then be compact.

\smallskip
\noindent
Fix $R>0$. For $\abs z\leq R$, polar coordinates and
$M_1(f,r)\leq M_{p_1}(f,r)$ give
\[
\begin{aligned}
\abs{P_t^+Q_Nf(z)}
&\leq
\frac{t}{\pi}
\int_{\abs w\geq N}
\abs{f(w)}
\e^{tR\abs w}
\e^{-t\abs w^2}\,dA(w)
\\
&\leq
2t\int_N^\infty
M_{p_1}(f,r)
\e^{-\frac{\alpha}{2}r^2}
\e^{-\frac{2t-\alpha}{2}r^2+tRr}r\,dr.
\end{aligned}
\]
H\"older's inequality in the radial variable, with the usual endpoint
interpretations, yields
\[
\sup_{\abs z\leq R}\abs{P_t^+Q_Nf(z)}
\lesssim
\norm{f}_{p_1,q_1,\alpha}\Lambda_R(N),
\]
where
\[
\Lambda_R(N)
:=
\norm{
\e^{-\frac{2t-\alpha}{2}r^2+tRr}
\chi_{[N,\infty)}(r)
}_{L^{q_1'}((0,\infty),r\,dr)}.
\]
Since $2t-\alpha>0$,
\(
\Lambda_R(N)\longrightarrow0\) as 
 \(N\to\infty\).
Consequently,
\begin{equation}
\label{eq:maximal-source-tail-uniform-local}
\lim_{N\to\infty}
\sup_{\norm{f}_{p_1,q_1,\alpha}\leq1}
\sup_{\abs z\leq R}
\abs{P_t^+Q_Nf(z)}
=0.
\end{equation}
Suppose that the asserted convergence in operator quasi-norm fails. Then
there exist $\varepsilon>0$, integers $N_j\to\infty$, and functions
$f_j$ such that
\[
\norm{f_j}_{p_1,q_1,\alpha}\leq1
\qquad\text{and}\qquad
\norm{P_t^+Q_{N_j}f_j}_{p_2,q_2,\beta}\geq\varepsilon.
\]
Since each $Q_{N_j}$ is a contraction, the sequence
\[
g_j:=Q_{N_j}f_j
\]
is bounded in $\Lpqa{\alpha}{p_1}{q_1}$. Moreover,
\eqref{eq:maximal-source-tail-uniform-local} gives
\[
P_t^+g_j\longrightarrow0
\qquad\text{locally uniformly on }\mathbb C.
\]
Since $P_t^+$ is compact and takes values in
$C(\mathbb C)$,
Corollary~\ref{cor:ambient-operator-compactness}\textup{(b)} yields
\[
\norm{P_t^+g_j}_{p_2,q_2,\beta}\longrightarrow0,
\]
contrary to the choice of $(f_j)_j$. This proves the required
convergence in operator quasi-norm and hence the compactness of $P_t$.
\end{proof}

\subsection{Sufficient conditions for \(P_t^+\)}
\label{subsec:compactness-sufficiency}

It remains to prove compactness of $P_t^+$ under the strict parameter
conditions. We approximate $P_t^+$ by the compact source truncations
$P_{t,N}^+$ and estimate the source tails by the annular discretization
and the tail bounds from
Proposition~\ref{prop:drifted-gaussian-operator}.

\begin{proposition}
\label{prop:compactness-sufficient}
Let
\(
1\leq p_1,q_1\leq\infty\) and \(
0<p_2,q_2\leq\infty\).
Suppose that either
\[
\Delta(t;\alpha,\beta)<0,
\]
or
\[
\Delta(t;\alpha,\beta)=0,
\qquad
q_1<q_2,
\qquad
d(p_2,q_2)<d(p_1,q_1).
\]
Then
\(
P_t^+:
\Lpqa{\alpha}{p_1}{q_1}
\longrightarrow
\Lpqa{\beta}{p_2}{q_2}
\)
is compact.
\end{proposition}

\begin{proof}
Retain the notation from Subsection~\ref{subsec:unified-shell-estimate}:
\[
p_2^\sharp:=\max\{1,p_2\},
\qquad
\rho_{12}:=\rho(p_1,p_2^\sharp),
\qquad
\sigma:=\frac1{q_1'}-\frac1{\rho_{12}},
\]
and
\[
\tau:=\frac{t}{\beta},
\qquad
\eta:=-\frac{\Delta(t;\alpha,\beta)}{2\beta}.
\]
Let \(f\in\Lpqa{\alpha}{p_1}{q_1}\), and put
\[
Y_k:=X_{p_1,q_1,\alpha}(\abs f;k),
\qquad k\geq0.
\]
Then
\[
\norm{Y}_{\ell^{q_1}}
=
\norm{f}_{p_1,q_1,\alpha}.
\]
Since
\[
X_{p_1,q_1,\alpha}(\abs{Q_Nf};k)=0,
\qquad k<N,
\]
Lemma~\ref{lem:unified-shell-estimate} and
Proposition~\ref{prop:annular-discretization}, with the defining
supremum used when \(q_2=\infty\), give
\[
\norm{P_t^+Q_Nf}_{p_2,q_2,\beta}
\lesssim
\norm{\mathcal T^{[N]}Y}_{\ell_{\mathrm{sh}}^{q_2}},
\]
where
\[
(\mathcal T^{[N]}Y)_\ell
:=
\sum_{k\geq N}
\e^{-\frac{\beta}{4}\abs{\ell-\tau k}^2}
\e^{-\eta k^2}
(1+k)^\sigma Y_k.
\]
Here and below, the implicit constants are independent of \(N\) and \(f\).  Therefore, 
\[
\norm{P_t^+-P_{t,N}^+}_{
\Lpqa{\alpha}{p_1}{q_1}
\to
\Lpqa{\beta}{p_2}{q_2}} = \norm{P_t^+Q_N}_{\Lpqa{\alpha}{p_1}{q_1}
\to
\Lpqa{\beta}{p_2}{q_2}}
\lesssim
\norm{\mathcal T^{[N]}}_{
\ell^{q_1}\to\ell_{\mathrm{sh}}^{q_2}}.
\]
The operator \(\mathcal T^{[N]}\) is precisely the
tail operator \(\mathcal T_{N-1}\) from
Proposition~\ref{prop:drifted-gaussian-operator} with
\[
\gamma_1=\frac{\beta}{4},
\qquad
\gamma_0=\eta.
\]

\smallskip
\noindent
\emph{The strict Gaussian region.}
If \(\Delta(t;\alpha,\beta)<0\), then \(\eta>0\).
Proposition~\ref{prop:drifted-gaussian-operator}\textup{(a)} gives
\[
\norm{\mathcal T^{[N]}}_{
\ell^{q_1}\to\ell_{\mathrm{sh}}^{q_2}}
\lesssim
\e^{-\frac{\eta}{2}(N-1)^2}
\lesssim
\e^{-\frac{\eta}{8}N^2},
\qquad N\geq2.
\]

\smallskip
\noindent
\emph{The Gaussian boundary.}
Suppose that
\(
\Delta(t;\alpha,\beta)=0.
\)
Then
\(\eta=0\).
The same computation as in the proof of
Proposition~\ref{prop:boundedness-sufficient}, now using the strict
conditions
\(
q_1<q_2\) and \(
d(p_2,q_2)<d(p_1,q_1)\),
gives
\[
\sigma+\frac1{q_2}<0.
\]
Hence, with
\(
\delta:=-\left(\sigma+\frac1{q_2}\right)>0,
\)
Proposition~\ref{prop:drifted-gaussian-operator}\textup{(b)} yields
\[
\norm{\mathcal T^{[N]}}_{
\ell^{q_1}\to\ell_{\mathrm{sh}}^{q_2}}
\lesssim
(1+N)^{-\delta}.
\]

\smallskip
\noindent
In either case,
\[
\norm{P_t^+-P_{t,N}^+}_{
\Lpqa{\alpha}{p_1}{q_1}
\to
\Lpqa{\beta}{p_2}{q_2}}
\longrightarrow0.
\]
Since \(\Lpqa{\beta}{p_2}{q_2}\) is complete and each \(P_{t,N}^+\) is compact by
Lemma~\ref{lem:compact-source-truncations}, the operator quasi-norm convergence 
\(
P_{t,N}^+\longrightarrow P_t^+
\)
implies that \(P_t^+\) is compact.
\end{proof}

\medskip
\noindent
We record two immediate consequences of
Theorem~\ref{thm:main-compactness}.

\begin{corollary}
\label{cor:equal-gaussian-compactness}
Let \(\alpha,t>0\) and
\(
0<p_1,q_1,p_2,q_2\leq\infty.
\)
Then
\(
P_t, P_t^+:
\Lpqa{\alpha}{p_1}{q_1}
\longrightarrow
\Lpqa{\alpha}{p_2}{q_2}
\)
are compact if and only if
\[
t=\alpha,
\qquad
p_1,q_1\geq1,
\qquad
q_1<q_2,
\qquad
d(p_2,q_2)<d(p_1,q_1).
\]
\end{corollary}

\begin{corollary}
\label{cor:equal-mixed-exponents-compactness}
Let \(\alpha,\beta,t>0\) and \(0<p,q\leq\infty\). Then
\(
P_t, P_t^+:
\Lpqa{\alpha}{p}{q}
\longrightarrow
\Lpqa{\beta}{p}{q}
\)
are compact if and only if
\[
p,q\geq1
\qquad\text{and}\qquad
\Delta(t;\alpha,\beta)<0.
\]
\end{corollary}

\begin{remark}
\label{rem:hou-liu-boundary}
We compare the classical Gaussian $L^p$-to-$L^q$ specialization of
Theorem~\ref{thm:main-compactness} with
\cite[Theorem~3.4]{HouLiu2023} in the case $n=1$. The maximal Fock
projection denoted by $Q_t$ in \cite{HouLiu2023} is precisely our
operator $P_t^+$. In the notation of that paper, let $a,b>0$ denote the source and target
Gaussian parameters and let $1\leq p,q<\infty$. For $c>0$, write
\[
dv_c(z):=\frac{c}{\pi}\e^{-c\abs z^2}\,dA(z).
\]
Thus the source and target spaces are
$L^p(\mathbb C,dv_a)$ and $L^q(\mathbb C,dv_b)$, respectively.
With our normalization,
\[
L^p(\mathbb C,dv_a)
=
L^{p,p}_{2a/p} 
\qquad \text{and} \qquad
L^q(\mathbb C,dv_b)
=
L^{q,q}_{2b/q}
\]
isometrically. Accordingly,
\[
pq\,
\Delta\left(t;\frac{2a}{p},\frac{2b}{q}\right)
=
pqt^2-4bpt+4ab.
\]
Thus the strict Gaussian condition in
\cite[Theorem~3.4]{HouLiu2023} agrees with the strict Gaussian region
in Theorem~\ref{thm:main-compactness}.

\smallskip
\noindent
At the Gaussian boundary, however, the corresponding diagonal
mixed-norm spaces satisfy
\[
d(p,p)=d(q,q)=0.
\]
Hence the strict defect condition in
Theorem~\ref{thm:main-compactness} cannot hold, and compactness is
excluded at the Gaussian boundary. This differs from the second
alternative in \cite[Theorem~3.4]{HouLiu2023}. The normalized exponential packets used in
Proposition~\ref{prop:compactness-necessary} give a direct
obstruction to compactness of $P_t$ at the Gaussian boundary.

\smallskip
\noindent
We also note an issue in the proof of the boundary case of
\cite[Theorem~3.4]{HouLiu2023}. There, boundedness of the maximal projection
$Q_t=P_t^+$ on the source space,
\[
Q_t:
L^p(\mathbb C,dv_a)
\longrightarrow
L^p(\mathbb C,dv_a)
\]
is invoked. By \cite[Theorem~2.6]{HouLiu2023}, applied with equal
source and target parameters and exponents, this boundedness is
equivalent to
\[
(pt-2a)^2=0,
\qquad\text{that is,}\qquad
pt=2a.
\]
Combining this identity with the Gaussian-boundary relation
\(
pqt^2-4bpt+4ab=0
\)
gives
\(
q=\frac{pb}{a},
\)
which contradicts the strict condition
\(
q<\frac{pb}{a}
\)
assumed in the second alternative of
\cite[Theorem~3.4]{HouLiu2023}. Thus the same-space boundedness used
in that argument is unavailable throughout the stated boundary
regime. This appears to account for the discrepancy.
\end{remark}

%================================================
\section{A transfer principle for Gaussian integral operators}
\label{sec:gaussian-twisted}
%================================================

The classifications proved above are not confined to the canonical
normalization of the Fock projection. The decisive observation is
that dilation and rotation of the output variable, combined with
multiplication by a real Gaussian factor, reduce a substantially
broader class of Gaussian integral operators to \(P_t\) and \(P_t^+\).
Under this reduction, the mixed-norm structure is preserved and the
effect of the transformation is absorbed into a modified target
Gaussian parameter. Theorem~\ref{thm:main-boundedness} therefore yields the corresponding
boundedness classification for this broader family, while
Theorem~\ref{thm:main-compactness} yields the compactness
classification.

\smallskip
\noindent
Integral operators generated by Gaussian and Fock kernels were
studied on classical Gaussian $L^p$-spaces by Dostani\'c and Zhu
\cite{DostanicZhu2008} and Furdui \cite{Furdui2008}. Hou and Liu subsequently studied the $L^p$-to-$L^q$ boundedness
and compactness of the canonical Fock and maximal Fock projections,
and obtained boundedness results for generalized Fock-kernel operators;
see \cite{HouLiu2023}. Liu, Shi, and Wei obtained the corresponding
boundedness classification in the Banach mixed-norm range
\cite[Section~5]{LiuShiWei2027}.
 The theorem below removes the Banach restriction, determines all
admissible quasi-Banach target cases, and gives the corresponding
compactness classification.
The reduction also shows that the full classification is governed by a
single transformed Gaussian balance condition, and it yields sharp
consequences for Berezin transforms and more general Fock-kernel
operators.

%------------------------------------------------
\medskip
\noindent
We begin with the elementary dilation--multiplier reduction underlying
the results of this section.

\begin{definition}
Let \(\lambda>0\), \(a\in\mathbb R\), and \(\omega\in\mathbb T\).
Define
\[
(U_{\lambda,a,\omega}f)(z)
:=
\e^{a\abs z^2}f(\lambda\omega z).
\]
\end{definition}

\begin{lemma}[Dilation--multiplier isomorphism]
\label{lem:dilation-multiplier}
Let \(\beta>0\), \(0<p,q\leq\infty\), and suppose that
\(
\beta>2a.
\)
Set
\[
\beta_{\lambda,a}
:=
\frac{\beta-2a}{\lambda^2}.
\]
Then
\(
U_{\lambda,a,\omega}:
L_{\beta_{\lambda,a}}^{p,q}
\longrightarrow
L_\beta^{p,q}\)
is a topological isomorphism, with
\[
U_{\lambda,a,\omega}^{-1}
=
U_{\lambda^{-1},-a/\lambda^2,\overline{\omega}}.
\]
When \(a=0\), it restricts to a topological isomorphism
\(
U_{\lambda,0,\omega}:
\mathcal F_{\beta/\lambda^2}^{p,q}
\longrightarrow
\mathcal F_\beta^{p,q}\).
\end{lemma}

\begin{proof}
Rotation invariance of the angular means gives
\[
M_p(U_{\lambda,a,\omega}f,r)
=
\e^{a \, r^2}M_p(f,\lambda r).
\]
If \(0<q<\infty\), the change of variables \(s=\lambda r\) yields
\[
\norm{U_{\lambda,a,\omega}f}_{p,q,\beta}^q
=
q\beta
\int_0^\infty
M_p(f,\lambda r)^q
\e^{-\frac{q(\beta - 2a)}2r^2}
r\,dr=
\frac{\beta}{\beta - 2a}
\norm{f}_{p,q,\beta_{\lambda,a}}^q.
\]
In the case \(q = \infty\), \( \norm{U_{\lambda,a,\omega}f}_{p, \infty,\beta} = \norm{f}_{p,\infty,\beta_{\lambda,a}}  \) follows directly from the definition. The
formula for the inverse is immediate. Finally, when \(a=0\), dilation
and rotation preserve holomorphy.
\end{proof}

%------------------------------------------------
\medskip
\noindent
We now apply this reduction to Gaussian integral operators in which
the target multiplier, the kernel coefficient, and the source Gaussian
parameter vary independently.

\smallskip
\noindent
Let
\(
a\in\mathbb R\), \(
b\in\mathbb C\setminus\{0\}\), and \(
c>0\). Initially on bounded compactly supported measurable functions, define 
\[
(T_{a,b,c}f)(z)
:=
\e^{a\abs z^2}
\int_{\C}
\e^{\operatorname{Re}(bz\overline w)-c\abs w^2}
f(w)\,dA(w)
\]
and
\[
(S_{a,b,c}f)(z)
:=
\e^{a\abs z^2}
\int_{\C}
\e^{bz\overline{w} -c\abs w^2}f(w) \,dA(w).
\]
Then
\[
\abs{S_{a,b,c}f}
\leq
T_{a,b,c}|f|.
\]
Write
\(
\lambda:=\frac{\abs b}{c}\) and \(
\omega:=\frac{b}{\abs b}\).
The definitions of \(P_c\), \(P_c^+\), and
\(U_{\lambda,a,\omega}\) give
\begin{equation}
\label{eq:gaussian-twisted-factorization}
S_{a,b,c}
=
\frac{\pi}{c}\,
U_{\lambda,  a,\omega}P_c,
\qquad
T_{a,b,c}
=
\frac{\pi}{c}\,
U_{\lambda, a,\omega}P_c^+.
\end{equation}

Let \(\alpha,\beta>0\) and assume that
\(
a<\frac{\beta}{2}.
\)
The transformed target Gaussian parameter is
\[
\beta_*
:=
\frac{\beta-2a}{\lambda^2}
=
\frac{c^2(\beta-2a)}{\abs b^2}.
\]
Define
\[
\Theta(a,b,c;\alpha,\beta)
:=
\abs b^2+(\alpha-2c)(\beta-2a).
\]
A direct calculation gives
\begin{equation}
\label{eq:gaussian-twisted-barrier}
\Delta(c;\alpha,\beta_*)
=
\frac{c^2}{\abs b^2}
\Theta(a,b,c;\alpha,\beta).
\end{equation}

\begin{theorem}[Gaussian-twisted Fock operators]
\label{thm:gaussian-twisted}
Let
\(
\alpha,\beta,c>0\), \(a \in \mathbb R\) such that \(
a<\frac{\beta}{2}\), \(
b\in\C\setminus\{0\}\),
and let
\(
0<p_1,q_1,p_2,q_2\leq\infty.
\)

\textup{(a)} The following assertions are equivalent:
\begin{enumerate}
\item[\textup{(i)}] The operator
\(
T_{a,b,c}:
L_\alpha^{p_1,q_1}
\longrightarrow
L_\beta^{p_2,q_2}
\)
is bounded.

\item[\textup{(ii)}]
The operator \(
S_{a,b,c}:
L_\alpha^{p_1,q_1}
\longrightarrow
L_\beta^{p_2,q_2}
\)
is bounded.

\item[\textup{(iii)}]
One has \(p_1,q_1\geq1\), and either
\[
\Theta(a,b,c;\alpha,\beta)<0,
\]
or
\[
\Theta(a,b,c;\alpha,\beta)=0,
\qquad
q_1\leq q_2,
\qquad
d(p_2,q_2)\leq d(p_1,q_1).
\]
\end{enumerate}

\item[\textup{(b)}] The following assertions are equivalent:
\begin{enumerate}[label=\textup{(\roman*)}]
\item The operator
\(
T_{a,b,c}:
\Lpqa{\alpha}{p_1}{q_1}
\longrightarrow
\Lpqa{\beta}{p_2}{q_2}
\)
is compact.
\item The operator
\(
S_{a,b,c}:
\Lpqa{\alpha}{p_1}{q_1}
\longrightarrow
\Lpqa{\beta}{p_2}{q_2}
\)
is compact.
\item One has $p_1,q_1\geq1$, and either
\[
\Theta(a,b,c;\alpha,\beta)<0,
\]
or
\[
\Theta(a,b,c;\alpha,\beta)=0,\qquad
q_1<q_2,\qquad
d(p_2,q_2)<d(p_1,q_1).
\]
\end{enumerate}
\end{theorem}

\begin{proof}
By Lemma~\ref{lem:dilation-multiplier},
\[
U_{\lambda,a,\omega}:
L_{\beta_*}^{p_2,q_2}
\longrightarrow
L_\beta^{p_2,q_2}
\]
is a topological isomorphism. Hence
\eqref{eq:gaussian-twisted-factorization} reduces the boundedness and
compactness of $T_{a,b,c}$ and $S_{a,b,c}$ to those of $P_c^+$ and
$P_c$, respectively, as operators
\[
P_c^+,\ P_c:
L_\alpha^{p_1,q_1}
\longrightarrow
L_{\beta_*}^{p_2,q_2}.
\]
Parts~\textup{(a)} and~\textup{(b)} therefore follow from
Theorems~\ref{thm:main-boundedness} and
\ref{thm:main-compactness}, respectively, together with
\eqref{eq:gaussian-twisted-barrier}.
\end{proof}

\begin{remark}
The phase of \(b\) is absorbed by the rotation
\(\omega= \frac{b}{\abs b}\), so the mapping conditions depend only on
\(\abs b\). When \(a\neq0\), the Gaussian multiplier
\(\e^{a\abs z^2}\) generally destroys holomorphy; thus
\(S_{a,b,c}\) is naturally regarded as an operator into the ambient
mixed-norm space.
\end{remark}

%------------------------------------------------

\medskip
\noindent
As a first consequence, consider the Berezin transform, a classical
tool in Fock-space function and operator theory; see
\cite{Zhu2012}.
For \(t>0\), define the Berezin transform by
\[
(B_tf)(z)
:=
\frac{t}{\pi}
\int_{\C}
\e^{-t\abs{z-w}^2}f(w)\,dA(w).
\]
Since
\[
B_t=\frac{t}{\pi}T_{-t,2t,t},
\qquad
\Theta(-t,2t,t;\alpha,\beta)
=
\alpha\beta+2\alpha t-2\beta t,
\]
Theorem~\ref{thm:gaussian-twisted} yields the following boundedness
and compactness classification. In the Banach range, the boundedness
part recovers \cite[Corollary~5.2]{LiuShiWei2027}.

\begin{corollary}[Berezin transforms]
\label{cor:berezin-transform}
Let $t,\alpha,\beta>0$ and
\(
0<p_1,q_1,p_2,q_2\leq\infty.
\)
Then the following assertions hold.

\begin{enumerate}[label=\textup{(\alph*)}]
\item The Berezin transform
\(
B_t:
\Lpqa{\alpha}{p_1}{q_1}
\longrightarrow
\Lpqa{\beta}{p_2}{q_2}
\)
is bounded if and only if $p_1,q_1\geq1$ and either
\[
\alpha\beta+2\alpha t-2\beta t<0,
\]
or
\[
\alpha\beta+2\alpha t-2\beta t=0,
\qquad
q_1\leq q_2,
\qquad
d(p_2,q_2)\leq d(p_1,q_1).
\]

\item The Berezin transform
\(
B_t:
\Lpqa{\alpha}{p_1}{q_1}
\longrightarrow
\Lpqa{\beta}{p_2}{q_2}
\)
is compact if and only if $p_1,q_1\geq1$ and either
\[
\alpha\beta+2\alpha t-2\beta t<0,
\]
or
\[
\alpha\beta+2\alpha t-2\beta t=0,
\qquad
q_1<q_2,
\qquad
d(p_2,q_2)<d(p_1,q_1).
\]
\end{enumerate}
\end{corollary}

\medskip
\noindent
Setting \(a=0\) gives the following Fock-kernel operators. For
\(c>0\) and \(b\in\C\setminus\{0\}\), define
\[
(P_{c,b}f)(z)
:=
\frac{c}{\pi}
\int_{\C}
\e^{bz\overline w-c\abs w^2}f(w)\,dA(w)
\quad \text{and} \quad
(P_{c,b}^+f)(z)
:=
\frac{c}{\pi}
\int_{\C}
\e^{\operatorname{Re}(bz\overline w)-c\abs w^2}
f(w)\,dA(w).
\]
Then
\[
P_{c,b}
=
\frac{c}{\pi}S_{0,b,c},
\qquad
P_{c,b}^+
=
\frac{c}{\pi}T_{0,b,c},
\qquad 
\Theta(0,b,c;\alpha,\beta)
=
\abs b^2-\beta(2c-\alpha).
\]
Theorem~\ref{thm:gaussian-twisted} gives the following result.

\begin{corollary}[Fock-kernel operators]
\label{cor:fock-kernel-operators}
Let
\(
\alpha,\beta,c>0\),
\(
b\in\C\setminus\{0\},
\)
and let
\(
0<p_1,q_1,p_2,q_2\leq\infty.
\)

\textup{(a)} The following assertions are equivalent:
\begin{enumerate}
\item[\textup{(i)}]
The operator
\(
P_{c,b}^+:
L_\alpha^{p_1,q_1}
\longrightarrow
L_\beta^{p_2,q_2}
\)
is bounded.

\item[\textup{(ii)}]
The operator
\(
P_{c,b}:
L_\alpha^{p_1,q_1}
\longrightarrow
L_\beta^{p_2,q_2}
\)
is bounded.

\item[\textup{(iii)}]
One has \(p_1,q_1\geq1\), and either
\[
\abs b^2<\beta(2c-\alpha),
\]
or
\[
\abs b^2=\beta(2c-\alpha),
\qquad
q_1\leq q_2,
\qquad
d(p_2,q_2)\leq d(p_1,q_1).
\]
\end{enumerate}

\textup{(b)} 
The following assertions are equivalent:
\begin{enumerate}[label=\textup{(\roman*)}]
\item The operator \(P_{c,b}^+: \Lpqa{\alpha}{p_1}{q_1} \to \Lpqa{\beta}{p_2}{q_2} \) is compact.
\item The operator \(P_{c,b}: \Lpqa{\alpha}{p_1}{q_1} \to \Lpqa{\beta}{p_2}{q_2}\) is compact.
\item One has \(p_1,q_1\geq1\), and either
\[
\abs b^2<\beta(2c-\alpha),
\]
or
\[
\abs b^2=\beta(2c-\alpha),
\qquad
q_1 < q_2,
\qquad
d(p_2,q_2) < d(p_1,q_1).
\]
\end{enumerate}
\end{corollary}

\medskip
\noindent
When \(b=c\), one has
\[
P_{c,b}=P_c,
\qquad
P_{c,b}^+=P_c^+,
\qquad
\abs b^2-\beta(2c-\alpha)
=
\Delta(c;\alpha,\beta).
\]
Thus Corollary~\ref{cor:fock-kernel-operators} recovers
Theorems~\ref{thm:main-boundedness} and
\ref{thm:main-compactness}.

\bigskip
\noindent 
\textbf{Funding.} X.~F. is supported by a grant from the National Science and Technology Council,
Taiwan (NSTC 114-2115-M-A49-003-MY3).

\bigskip
\noindent
\textbf{Declaration of competing interests.} The authors declare that they have no competing interests.

\bigskip
\noindent
\textbf{Data availability.}
No data were used for the research described in this article.

\bigskip
\noindent
\textbf{Use of artificial intelligence.}
The authors used artificial-intelligence tools for language editing,
\LaTeX\ formatting, and limited assistance with local mathematical
reasoning.  The proof strategy and mathematical development are the
authors' own; all mathematical content was independently written and checked by the authors, who take full responsibility for it.

%================================================

%================================================
\end{document}